\documentclass{amsart}

\usepackage[final,ulem=normalem,ulem=normalbf]{changes}
\setauthormarkup{}

\usepackage[english]{babel}
\usepackage[utf8]{inputenc}

\usepackage{amsthm,amsfonts,amssymb,amsmath}
\usepackage{stmaryrd}
\numberwithin{equation}{section}
\usepackage{dsfont}

\usepackage{url}
\usepackage{algorithmic,algorithm}
\usepackage{graphics}
\usepackage{epsfig}
\usepackage{epstopdf}
\usepackage{psfrag}               
\usepackage{mathrsfs}
\usepackage{mathtools}            
\usepackage{subfigure}
\usepackage{longtable}	

\usepackage{paralist}

\usepackage[pdftex,
            pdftitle={Covariance structure of SPDEs with multiplicative Levy noise},
            pdfauthor={K. Kirchner and A. Lang and S. Larsson},
            bookmarksopen,
            colorlinks,
            linkcolor=black,
            urlcolor=black,
	    citecolor=black
]{hyperref}
\usepackage{hyperref}

\newcommand{\bbB}{\mathbb{B}}

\newcommand{\bbE}{\mathbb{E}}

\newcommand{\bbM}{\mathbb{M}}
\newcommand{\bbN}{\mathbb{N}}
\newcommand{\bbP}{\mathbb{P}}

\newcommand{\bbR}{\mathbb{R}}
\newcommand{\bbT}{\mathbb{T}}

\newcommand{\cA}{\mathcal{A}}
\newcommand{\cB}{\mathcal{B}}

\newcommand{\cD}{\mathcal{D}}

\newcommand{\cF}{\mathcal{F}}

\newcommand{\cH}{\mathcal{H}}
\newcommand{\cI}{\mathcal{I}}

\newcommand{\cL}{\mathcal{L}}

\newcommand{\cP}{\mathcal{P}}

\newcommand{\cW}{\mathcal{W}}
\newcommand{\cX}{\mathcal{X}}
\newcommand{\cY}{\mathcal{Y}}
\newcommand{\cZ}{\mathcal{Z}}

\newcommand{\scrB}{\mathscr{B}}

\newcommand{\Hr}{\dot{H}^{r}}
\newcommand{\lhs}{\cL_{2}}
\newcommand{\lnn}{\cL_{1}}
\newcommand{\lnp}{\cL_{1}^{+}}

\newcommand{\clp}{\cL_{p}}
\newcommand{\cov}{\operatorname{Cov}}
\newcommand{\tr}{\operatorname{tr}}
\newcommand{\rd}{\mathrm{d}}

\newcommand{\xH}{\mathbin{\hat{\otimes}}}
\newcommand{\xpi}{\mathbin{\hat{\otimes}_{\pi}}}
\newcommand{\xeps}{\mathbin{\hat{\otimes}_{\varepsilon}}}

\newcommand{\cBpi}[2]{\widetilde{\cB}^{(\pi)}( #1, #2)} 

\newcommand{\uip}[2]{\langle #1, #2\rangle_U}
\newcommand{\utip}[2]{\langle #1, #2\rangle_{U^{(2)}}}
\newcommand{\chip}[2]{\langle #1, #2\rangle_{\cH}}
\newcommand{\chnorm}[1]{\| #1 \|_{\cH}}

\newcommand{\hip}[2]{\langle #1, #2\rangle_H}
\newcommand{\htip}[2]{\langle #1, #2\rangle_{H^{(2)}}}
\newcommand{\lthip}[2]{\langle #1, #2\rangle_{L^2(\bbT; H)}}
\newcommand{\hnorm}[1]{\| #1 \|_{H}}

\newcommand{\vip}[2]{\langle #1, #2\rangle_{V}}

\newcommand{\vvdp}[2]{\,_{V^*}\langle #1, #2\rangle_{V}}
\newcommand{\vdpt}[2]{\,_V\langle #1, #2\rangle_{V^*}}

\newcommand{\vnorm}[1]{\| #1 \|_{V}}
\newcommand{\vdnorm}[1]{\| #1 \|_{V^{*}}}

\newcommand{\xnorm}[1]{\| #1 \|_{\cX}}

\newcommand{\ynorm}[1]{\| #1 \|_{\cY}}

\definecolor{green2}{RGB}{0,128,0}
\definecolor{red2}{RGB}{230,0,0}
\definecolor{grey}{RGB}{195,195,195}
\definecolor{orange}{RGB}{255,100,0}

\graphicspath{{Images/}}

\newtheorem{lemma}{Lemma}[section]
\newtheorem{proposition}[lemma]{Proposition}
\newtheorem{theorem}[lemma]{Theorem}

\theoremstyle{remark}
\newtheorem{remark}[lemma]{Remark}

\theoremstyle{definition}
\newtheorem{definition}[lemma]{Definition}
\newtheorem{assumption}[lemma]{Assumption}

\definecolor{darkgreen}{rgb}{0,.6,0}

\definecolor{darkred}{rgb}{.6,0,0}

\usepackage{cancel}
\usepackage{ulem}

\begin{document}

\title[Covariance structure of SPDEs with multiplicative L\'evy noise]%
	{Covariance structure of parabolic stochastic partial differential equations with multiplicative L\'evy noise}

\author{Kristin Kirchner}

\address{Kristin Kirchner\\
Department of Mathematical Sciences\\
Chalmers University of Technology and University of Gothenburg\\
SE-412 96 Gothenburg, Sweden}

\email{kristin.kirchner@chalmers.se}

\author{Annika Lang}

\address{Annika Lang\\
Department of Mathematical Sciences\\
Chalmers University of Technology and University of Gothenburg\\
SE-412 96 Gothenburg, Sweden}

\email{annika.lang@chalmers.se}

\author{Stig Larsson}

\address{Stig Larsson\\
Department of Mathematical Sciences\\
Chalmers University of Technology and University of Gothenburg\\
SE-412 96 Gothenburg, Sweden}

\email{stig@chalmers.se}

	
\thanks{
Acknowledgement. This work was supported in part by the Knut and Alice Wallenberg Foundation as well as the Swedish Research Council under Reg.~No.~621-2014-3995.
The authors thank Roman Andreev, Arnulf Jentzen, and Christoph Schwab for fruitful discussions and helpful comments.
}


\begin{abstract}
The characterization of the covariance function of the solution
process to a stochastic partial differential equation is considered in 
the parabolic case with multiplicative L\'evy noise of affine type.
For the second moment of the mild solution, a well-posed
deterministic space-time variational problem posed on projective and injective
tensor product spaces is derived, which subsequently leads to a deterministic equation
for the covariance function.
\end{abstract}



\date{\today}

\maketitle

\section{Introduction}\label{section:intro}

The covariance function of a stochastic process is an interesting quantity for
the following reasons: It provides information about the correlation of the process
with itself at pairs of time points.
In addition, it shows if this relation is stationary, i.e., whether or
not it changes when shifted in time, and if it follows a trend.
In~\cite{lang2013} the covariance of the solution process to a
parabolic stochastic partial differential equation driven by an
additive $Q$-Wiener process has been described as the solution to a
deterministic, tensorized evolution equation. In this case the
  solution process is also Gaussian with mean zero and therefore
  completely characterized by its covariance. It is now natural
  to ask whether it is possible to establish such an equation
also for covariance functions of solutions to stochastic partial
differential equations driven by multiplicative noise.

In the present paper, we extend the study of the
  covariance function to solution processes of parabolic stochastic
  partial differential equations driven by multiplicative L\'evy noise
  in the framework of~\cite{peszat2007}.  In this case the
  solution process is no longer fully characterized by the covariance,
  but the covariance function is still of interest as mentioned above.
  We emphasize that it is the extension to multiplicative noise which
  is the main motivation and challenge here; the extension to L\'evy
  noise is rather straightforward since the theory of the
  corresponding It\^o integral is more or less parallel to the Wiener
  case. The multiplicative operator is assumed to be affine.
Clearly, under appropriate assumptions on the driving L\'evy process,
the mean function of the mild solution satisfies the corresponding
deterministic, parabolic evolution equation as in the case of additive
Wiener noise, since in both cases the stochastic integral has
  expectation zero. However, the presence of a multiplicative term
changes the behavior of the second moment and the covariance.  We
prove that also in this case the second moment as well as the
covariance of the square-integrable mild solution satisfy
deterministic space-time variational problems posed on tensor products
of Bochner spaces.  In contrast to the case of additive Wiener noise
considered in~\cite{lang2013}, the trial and the test space are not
Hilbert tensor products. Instead we use different notions of tensor
product spaces to obtain well-posed variational problems. These tensor
product spaces are non-reflexive Banach spaces.  In addition, the
resulting bilinear form in the variational problem does not
arise from taking the tensor of the corresponding deterministic
parabolic operator with itself, but it involves a non-separable
operator mapping to the dual space of the test space.  For these
reasons, well-posedness of the derived deterministic variational
problems is not an immediate consequence, and operator theory
  on the tensor product spaces is used to derive it. We
  emphasize that, although the present manuscript is
  rather abstract, numerical methods based on this variational
  problem are currently under investigation. 

The structure of the present paper is as follows: In Section~\ref{section:spde} we
present the parabolic stochastic differential equation and its mild solution,
whose covariance function we aim to describe.
The results formulated in Section~\ref{section:aux} will be used
for proving the main results of this paper
in Sections~\ref{section:moment}--\ref{section:covariance}.
More precisely, in Subsections~\ref{subsec:tensor}--\ref{subsec:diagonal}
we present different notions of tensor product spaces
and several operators arising in the variational problems satisfied
by the second moment and the covariance of the mild solution. The weak
It\^o isometry, which we introduce in Subsection~\ref{subsec:weakito},
is crucial for the derivation of the deterministic variational problems.
Theorems~\ref{thm:deterministic} and~\ref{thm:wellposed} in Sections~\ref{section:moment}
and~\ref{section:exuni} provide the main results of this paper:
In Theorem~\ref{thm:deterministic} we show that
the second moment of the mild solution satisfies a deterministic space-time variational
problem posed on non-reflexive tensor product spaces. In order to be able
to formulate this variational problem, we need some additional regularity of the
second moment which we prove first.
The aim of Section~\ref{section:exuni} is to establish well-posedness of
the derived variational problem. Since the variational problem
is posed on non-reflexive Banach spaces, it is not possible to
apply standard inf-sup theory to achieve this goal.
Instead, we show that the operator associated with the bilinear
form appearing in the variational problem is bounded from below,
which implies uniqueness of the solution to the variational problem.
Finally, in Section~\ref{section:covariance} we use the results of
the previous sections to obtain a well-posed space-time variational
problem satisfied by the covariance function of the mild solution.

\section{The stochastic partial differential equation}\label{section:spde}

In this section the investigated stochastic partial differential equation as well as
the setting that we impose on it are presented. In addition, we formulate the definition
as well as existence, uniqueness, and regularity results of the so-called mild solution to
this equation in Definition~\ref{def:mildsol} and Theorem~\ref{thm:exunimild}.
Finally, in Lemma~\ref{lem:weakequal} we state a property of the mild solution
which will be essential for the derivation of the deterministic equation satisfied
by its second moment and its covariance function
in Sections~\ref{section:moment} and~\ref{section:covariance}, respectively.

For two Banach spaces $E_1$ and $E_2$ we denote by $\cL(E_1; E_2)$
the space of bounded linear operators mapping from $E_1$ to $E_2$.
In addition, we write $\clp(H_1; H_2)$ for the space of Schatten class
operators of $p$-th order
mapping from $H_1$ to $H_2$, where $H_1$ and $H_2$ are separable Hilbert spaces.
Here, for $1 \leq p < \infty$ an operator $T\in \cL(H_1;H_2)$
is called a \emph{Schatten-class operator of $p$-th order},
if $T$ has a finite $p$-th Schatten norm, i.e.,
\[
    \| T\|_{\clp(H_1;H_2)} := \left( \sum\limits_{n\in\bbN} s_n(T)^p \right)^{\frac{1}{p}} < +\infty,
\]
where $s_1(T) \geq s_2(T) \geq \ldots \geq s_n(T) \geq \ldots \geq 0$ are
the singular values of $T$, i.e., the eigenvalues of the operator $(T^* T)^{1/2}$
and $T^* \in \cL(H_2; H_1)$ denotes the adjoint of $T$.
If $H_1=H_2=H$ we abbreviate $\clp(H; H)$ by $\clp(H)$.
For the case $p=1$ and a separable Hilbert space $H$ with inner product $\hip{\cdot}{\cdot}$
and orthonormal basis $(e_n)_{n\in\bbN}$ we introduce the \emph{trace}
of an operator $T\in\lnn(H)$ by
\[
   \tr(T) := \sum\limits_{n\in\bbN} \hip{Te_n}{e_n}.
\]
The trace $\tr(T)$ is independent of the choice of the orthonormal basis
and it satisfies $|\tr(T)| \leq \|T\|_{\lnn(H)}$, cf.~\cite[Proposition~C.1]{daprato1992}.
By $\lnp(H)$ we denote the space of all nonnegative, symmetric trace class operators on $H$, i.e.,
\[
   \lnp(H) := \left\{ T \in \lnn(H) : \hip{T\varphi}{\varphi} \geq 0,
                   \hip{T\varphi}{\psi} = \hip{\varphi}{T\psi} \quad \forall \varphi,\psi \in H \right\}.
\]
For $p=2$, the norm $\| T\|_{\lhs(H_1;H_2)}$ coincides with the \emph{Hilbert--Schmidt} norm.

In the following $U$ and $H$ denote separable Hilbert spaces with
norms $\|\cdot\|_U$ and $\|\cdot\|_H$ induced by the inner products
$\uip{\cdot}{\cdot}$ and $\hip{\cdot}{\cdot}$, respectively.

Let $L := (L(t), t\geq 0)$ be an adapted, square-integrable, $U$-valued L\'evy process
defined on a complete filtered probability space $(\Omega, \cA, (\cF_t)_{t\geq 0}, \bbP)$.
More precisely, we assume that
\begin{enumerate}
  \item\label{def:levy-i}
    $L$ has independent increments, i.e., for all
    $0\leq t_0 < t_1 < \ldots < t_n$ the $U$-valued random variables
    $L(t_1)-L(t_0)$, $L(t_2)-L(t_1)$, $\ldots$, $L(t_n) - L(t_{n-1})$ are independent;
  \item\label{def:levy-ii}
    $L$ has stationary increments, i.e., the distribution of $L(t)-L(s)$,
    $s\leq t$, depends only on the difference $t-s$;
  \item\label{def:levy-iii}
    $L(0) = 0$ $\bbP$-almost surely;
  \item\label{def:levy-iv}
    $L$ is stochastically continuous, i.e.,
    \[
      \lim\limits_{\substack{s\to t \\ s\geq 0}} \bbP(\|L(t)-L(s)\|_{U} > \epsilon) = 0
      \quad \forall \epsilon > 0, \quad \forall t \geq 0;
    \]
  \item\label{def:levy-v}
    $L$ is adapted, i.e., $L(t)$ is $\cF_t$-measurable for all $t\geq 0$;
  \item\label{def:levy-vi}
    $L$ is square-integrable, i.e., $\bbE \left[\| L(t)\|_U^2 \right] < +\infty$ for all $t\geq 0$.
\end{enumerate}

Furthermore, we assume that for $t>s\geq 0$ the increment $L(t)-L(s)$ is independent of $\cF_s$
and that $L$ has zero mean and covariance operator $Q\in \lnp(U)$, i.e.,
for all $s,t\geq 0$ and $x, y \in U$ it holds: $\bbE \uip{L(t)}{x} = 0$ and
\begin{equation}\label{eq:levyQ}
\bbE \left[ \uip{L(s)}{x}\uip{L(t)}{y} \right] = \min\{s,t\} \, \uip{Q x}{y},
\end{equation}
cf.~\cite[Theorem~4.44]{peszat2007}. Note that under these assumptions, the L\'evy process
$L$ is a martingale with respect to the filtration $(\cF_t)_{t\geq 0}$ by~\cite[Proposition~3.25]{peszat2007}.

In addition, since $Q\in\lnp(U)$ is a nonnegative, symmetric trace class operator,
there exists an orthonormal
eigenbasis $(e_n)_{n\in\bbN} \subset U$ of $Q$
with corresponding eigenvalues $(\gamma_n)_{n\in\bbN} \subset \bbR_{\geq 0}$, i.e.,
$Qe_n = \gamma_n e_n$ for all $n\in\bbN$, and for $x\in U$ we may define
the fractional operator $Q^{1/2}$ by
\[
    Q^{\frac{1}{2}}  x := \sum\limits_{n\in\bbN} \gamma_n^{\frac{1}{2}} \uip{x}{e_n} \, e_n
\]
as well as its pseudo inverse $Q^{-1/2}$ by
\[
    Q^{-\frac{1}{2}} x := \sum\limits_{n\in\bbN \, : \, \gamma_n \neq 0} \gamma_n^{-\frac{1}{2}} \uip{x}{e_n} \, e_n.
\]
We introduce the vector space $\cH := Q^{1/2} U$.
Then $\cH$ is a Hilbert space with respect to the inner product
$\chip{\cdot}{\cdot} := \uip{Q^{-1/2} \cdot}{Q^{-1/2} \cdot}$.

Furthermore, let $A\colon \cD(A) \subset H \to H$ be a densely defined,
self-adjoint, positive definite linear operator,
which is not necessarily bounded, but which has a compact inverse.
In this case $-A$ is the generator of an analytic semigroup of contractions
$(S(t), t\geq 0)$
and for $r \geq 0$ the fractional power operator $A^{r/2}$
is well-defined on a domain $\cD(A^{r/2}) \subset H$,
cf.~\cite[Chapter~2]{pazy1983}. 
We define the Hilbert space $\Hr$ as the completion of $\cD(A^{r/2})$ equipped with the inner product
\[
    \langle \varphi, \psi\rangle_{\Hr} := \hip{ A^{r/2} \varphi}{A^{r/2} \psi}
\]
and obtain a scale of Hilbert spaces with $\dot{H}^s \subset \Hr \subset \dot{H}^0 = H$
for $0\leq r \leq s$. Its role is to measure spatial regularity.
We denote the special case when $r=1$ by $V:=\dot{H}^1$.
In this way we obtain a Gelfand triple
\[
    V \hookrightarrow H \cong H^* \hookrightarrow V^*,
\]
%
where we use $\,^*$ to denote the identification of the dual spaces of $H$ and $V$
with respect to the pivot space $H$. Later on, the notation $\,^\prime$ will be used when
addressing to the dual space in its classical sense, i.e., as the space of all linear continuous mappings to $\bbR$.
In addition, although the operator $A$ is assumed to be self-adjoint, we
denote by $A^*\colon V\to V^*$ its adjoint for clarification whenever we consider
the adjoint instead of the operator itself.
With these definitions, the operator $A$ and its adjoint are bounded,
i.e.,~$A, A^* \in \cL(V; V^*)$, since for $\varphi, \psi \in V$ it holds
\[
    \vvdp{A \varphi}{\psi} = \hip{ A^{1/2} \varphi}{A^{1/2} \psi} = \vip{\varphi}{\psi} = \vdpt{\varphi}{A^* \psi},
\]
where $\vvdp{\cdot}{\cdot}$ and $\vdpt{\cdot}{\cdot}$ denote dual pairings between $V$ and $V^*$.

We consider the stochastic partial differential equation

\begin{equation}\label{eq:spdeall}
\begin{split}
\rd X(t) + A X(t) \, \rd t &= G(X(t)) \, \rd L(t), \quad t \in \bbT := [0,T], \\
X(0) &= X_0,
\end{split}
\end{equation}
for finite~$T>0$. In order to obtain existence and uniqueness of a solution to this problem as well as
additional regularity for its second moment, which will be needed later on,
we impose the following assumptions on the initial value~$X_0$ and the operator~$G$.
\begin{assumption}\label{ass:spde}
The initial value~$X_0$ and the operator~$G$ in~\eqref{eq:spdeall} satisfy:
\begin{enumerate}
\item \label{ass:spde-i} $X_0$ is a square-integrable, $H$-valued random variable, i.e., $X_0 \in L^2(\Omega; H)$, which is $\cF_0$-measurable.
\item \label{ass:spde-ii} $G\colon H \to \lhs(\cH; H)$ is an affine operator, i.e.,
$G(\varphi) = G_1(\varphi) + G_2$ with operators $G_1 \in \cL(H, \lhs(\cH; H))$
and $G_2 \in \lhs(\cH; H)$.
\item \label{ass:spde-iii} There exists a regularity exponent $r\in [0,1]$ such that
$X_0 \in L^2(\Omega; \Hr)$ and $A^{r/2} S(\cdot) G_1 \in L^2(\bbT; \cL(\Hr; \lhs(\cH; H)))$,
i.e.,
\[
     \int_0^T \| A^{\frac{r}{2}} S(t) G_1 \|_{\cL(\Hr; \lhs(\cH; H))}^2 \, \rd t < +\infty.
\]
\item \label{ass:spde-iv} $A^{1/2} S(\cdot) G_1 \in L^2(\bbT; \cL(\Hr; \lhs(\cH; H)))$, i.e.,
\[
     \int_0^T \| A^{\frac{1}{2}} S(t) G_1 \|_{\cL(\Hr; \lhs(\cH; H))}^2 \, \rd t < +\infty,
\]
with the same value for $r\in[0,1]$ as in~\ref{ass:spde-iii}.
\item \label{ass:spde-v} $G_1 \in \cL(V, \cL(U; H))$ and $G_2 \in \cL(U; H)$.
\end{enumerate}
\end{assumption}

Note that the assumption on $G_1$ in part~\ref{ass:spde-iv} implies the one in part~\ref{ass:spde-iii}.
Conditions~\ref{ass:spde-i}--\ref{ass:spde-iii} guarantee $\dot{H}^r$ regularity of the mild solution
(cf.~Theorem~\ref{thm:exunimild}),
but we need all five assumptions for our main results
in Sections~\ref{section:moment} and~\ref{section:covariance}.

Before we derive the deterministic variational problems satisfied by
the second moment and the covariance of the solution $X$ to~\eqref{eq:spdeall}
in Sections~\ref{section:moment} and~\ref{section:covariance},
we have to specify which kind of solvability we consider.
In addition, existence and uniqueness of this solution must be guaranteed.

\begin{definition}\label{def:mildsol}
    A predictable process $X\colon\Omega \times \bbT \to H$ is called a mild solution to~\eqref{eq:spdeall},
    if $\sup\nolimits_{t\in\bbT} \| X(t)\|_{L^2(\Omega; H)} < +\infty$ and
    \begin{equation}\label{eq:mildsol}
        X(t) = S(t) X_0 + \int_0^t S(t-s) G(X(s)) \, \rd L(s), \quad t \in \bbT.
    \end{equation}
\end{definition}

It is a well-known result that there exists a unique mild solution
to equations driven by affine multiplicative noise as considered above.
More precisely, we have the following theorem.

\begin{theorem}\label{thm:exunimild} 
    Under Assumption~\ref{ass:spde}~\ref{ass:spde-i}--\ref{ass:spde-ii} there exists (up to modification)
    a unique mild solution $X$ of~\eqref{eq:spdeall}.
    If additionally Condition~\ref{ass:spde-iii} of Assumption~\ref{ass:spde} holds,
    then the mild solution satisfies
    \[
        \sup\limits_{t\in\bbT} \| X(t)\|_{L^2(\Omega; \Hr)} < +\infty,
    \]
    i.e., $X \in L^{\infty}\bigl( \bbT; L^2(\Omega; \Hr) \bigr)$.
    \end{theorem}

\begin{proof}
    The first part of the theorem follows from~\cite[Theorem~9.29]{peszat2007}.
    Suppose now that condition~\ref{ass:spde-iii} is satisfied.
    By the dominated convergence theorem the sequence of integrals
    \[
        \int_0^{T} \| A^{\frac{r}{2}} S(\tau) G_1 \|_{\cL(\Hr; \lhs(\cH; H))}^2 \,
        \mathds{1}_{(0,T/n)}(\tau) \, \rd \tau ,
    \]
    where $n\in\bbN$ and $\mathds{1}_{(0,T/n)}$ denotes the indicator function on the interval $(0,T/n)$,
    converges to zero as $n\to\infty$. Therefore, there exists $\widetilde{T} \in (0,T]$ such that
    \[
        \kappa^2 := \int_0^{\widetilde{T}} \| A^{\frac{r}{2}} S(\tau) G_1 \|_{\cL(\Hr; \lhs(\cH; H))}^2 \, \rd \tau < 1.
    \]
    Define $\widetilde{\bbT} := \bigl[0,\widetilde{T} \bigr]$,
    $\cZ := L^{\infty}\bigl( \widetilde{\bbT}; L^2(\Omega; \Hr) \bigr)$ and
    \[
        \Upsilon \colon \cZ \to \cZ, \quad \Upsilon(Z)(t) := S(t)X_0 + \int_0^t S(t-s) G(Z(s)) \, \rd L(s), \quad t \in \widetilde{\bbT} .
    \]
    Then $\Upsilon$ is a contraction: For every $t \in \widetilde{\bbT}$ and $Z_1, Z_2 \in \cZ$ we have
    \begin{align*}
    \| \Upsilon(Z_1)(t) &- \Upsilon(Z_2)(t) \|_{L^2(\Omega; \Hr)}^2
         = \bbE \Bigl\| \int_0^t S(t-s) G_1(Z_1(s) - Z_2(s)) \, \rd L(s) \Bigr\|_{\Hr}^2 \\
        &= \bbE \Bigl\| \int_0^t A^{\frac{r}{2}} S(t-s) G_1(Z_1(s) - Z_2(s)) \, \rd L(s) \Bigr\|_H^2, \\
    \intertext{since $A$ and, hence, $A^{r/2}$ are closed operators.
    Now the application of It\^o's isometry for the case of a L\'evy process, cf.~\cite[Corollary~8.17]{peszat2007}, yields }
        &= \bbE \int_0^t \| A^{\frac{r}{2}} S(t-s) G_1(Z_1(s) - Z_2(s)) \|_{\lhs(\cH; H)}^2 \, \rd s \\
        &\leq \bbE \int_0^t \| A^{\frac{r}{2}} S(t-s) G_1 \|_{\cL(\Hr; \lhs(\cH; H))}^2
                                \| Z_1(s) - Z_2(s)\|_{\Hr}^2 \, \rd s \\
        &=\int_0^t  \| A^{\frac{r}{2}} S(t-s) G_1 \|_{\cL(\Hr; \lhs(\cH; H))}^2 \bbE\left[  \| Z_1(s) - Z_2(s)\|_{\Hr}^2 \right] \, \rd s, 
        \end{align*}
    where the interchanging of the expectation and the time integral is justified by Tonelli's theorem. Therefore, we obtain the estimate
    \[
        \| \Upsilon(Z_1)(t) - \Upsilon(Z_2)(t) \|_{L^2(\Omega; \Hr)}^2
            \leq \kappa^2 \sup\limits_{s\in\widetilde{\bbT}} \bbE \| Z_1(s) - Z_2(s)\|_{\Hr}^2
    \]
    for all $t\in\bbT$
    and $\| \Upsilon(Z_1) - \Upsilon(Z_2) \|_{\cZ} \leq \kappa \|Z_1 - Z_2\|_{\cZ}$,
    which shows that $\Upsilon$ is a contraction.
    By the Banach fixed point theorem,
    there exists a unique fixed point $X_{*}$ of $\Upsilon$ in $\cZ$.
    Hence, $X = X_{*}$ is the unique mild solution to~\eqref{eq:spdeall} on $\widetilde{\bbT}$ and
    \[
        \|X\|_{\cZ}^2 = \sup\limits_{t\in\widetilde{\bbT}} \bbE \| X(t)\|_{\Hr}^2 < +\infty.
    \]
    The claim of the theorem follows from iterating the same argument on the intervals
    \[
        \bigl[(m-1)\widetilde{T}, \, \min\{m\widetilde{T}, T\} \bigr],
        \quad m \in \bigl\{ 1,2, \ldots, \bigl\lceil T/\widetilde{T} \bigr\rceil \bigr\}. \qedhere
    \]
\end{proof}

Lemma~\ref{lem:weakequal} relates the concepts of weak and mild solutions
of stochastic partial differential equations, cf.~\cite[Section~9.3]{peszat2007},
and provides the basis for establishing the connection between
the second moment of the mild solution and a space-time variational problem.
In order to state it, we first have to define the differential operator $\partial_t$
and the weak stochastic integral.
For a vector-valued function $u\colon\bbT\to~H$ taking values in a Hilbert space $H$
we define the distributional derivative $\partial_t u$ as the $H$-valued distribution satisfying
\[
    \hip{(\partial_t u)(w)}{\varphi} = -\int_0^T \frac{\rd w}{\rd t} (t) \hip{u(t)}{\varphi} \, \rd t 
\]
for all $\varphi \in H$ and $w \in C_0^{\infty}(\bbT; \bbR)$, cf.~\cite[Definition~3 in~§XVIII.1]{dautray1999}.

In the following, we consider the spaces $L^2(\Omega\times\bbT; \lhs(\cH;H))$
as well as
$L^2(\Omega\times\bbT; \cL(U;H))$
of square-integrable functions taking values in $\lhs(\cH;H)$ and $\cL(U;H)$, respectively,
with respect to the measure space
$(\Omega\times\bbT, \cP_\bbT, \bbP\otimes\lambda)$, where $\cP_\bbT$ denotes the
$\sigma$-algebra of predictable subsets of $\Omega\times\bbT$
and $\lambda$ the Lebesgue measure on $\bbT$.
For a predictable process $\Phi \in L^2(\Omega\times\bbT; \lhs(\cH;H))$ and
a continuous $H$-valued function $v\in C^0(\bbT;H)$
we define
the stochastic process $\Psi \in L^2(\Omega \times \bbT; \lhs(\cH; \bbR))$ by
\[
    \Psi(t) \colon z \mapsto \hip{v(t)}{\Phi(t) z} \quad \forall z\in \cH,
\]
for all $t\in\bbT$.
The predictability of $\Psi$ follows from the continuity of $v$ on $\bbT$
and the predictability of $\Phi$. 

The weak stochastic integral $\int\nolimits_0^T \hip{v(t)}{\Phi(t) \, \rd L(t)}$
is then defined as the stochastic integral with respect to the integrand $\Psi$, i.e.,
\begin{align}\label{eq:weakstochint}
    \int_0^T \hip{v(t)}{\Phi(t) \, \rd L(t)} := \int_0^T  \Psi(t) \, \rd L(t)
    \quad \bbP\text{-a.s.},
\end{align}
cf.~\cite[p.~151]{peszat2007}.
Its properties imply by~\cite[Equation~(9.20)]{peszat2007} the following lemma.

\begin{lemma}\label{lem:weakequal}
    Let Assumption~\ref{ass:spde}~\ref{ass:spde-i}--\ref{ass:spde-ii}
    be satisfied and let
    $X$ be the mild solution to~\eqref{eq:spdeall}.
    Then it holds $\bbP$-almost surely that
    \begin{equation*}
        \lthip{X}{(-\partial_t + A^*) v} = \hip{X_0}{v(0)} + \int_0^T \hip{v(t)}{G(X(t)) \, \rd L(t)}
    \end{equation*}
    for all $v \in C^1_{0,\{T\}}(\bbT; \cD(A^*)) := \{ w \in C^1(\bbT, \cD(A^*)) : w(T)=0\}$.
\end{lemma}



\section{Auxiliary results}\label{section:aux}

The aim of this section is to prove some auxiliary results that
will be needed later on to derive the main results
in Sections~\ref{section:moment},~\ref{section:exuni}, and~\ref{section:covariance}.

In Subsection~\ref{subsec:tensor} we introduce
different notions of tensor product spaces and some of their properties.
The deterministic equations satisfied by
the second moment and the covariance will be posed on these kinds of spaces.

Next, in Subsection~\ref{subsec:covkern}, we use these tensor product spaces
to define the covariance kernel associated with the driving L\'evy
process $L$ and derive some additional results for the interaction
of this covariance kernel with the operators $G_1$ and $G_2$, see
Lemmas~\ref{lem:ggupi}~and~\ref{lem:ggqvv}.

In order to formulate our main results in Sections~\ref{section:moment}--\ref{section:covariance}
in a compact way, we introduce two operators in Subsection~\ref{subsec:diagonal}.
These operators appear in the deterministic equations
in Sections~\ref{section:moment}~and~\ref{section:covariance}
and the results of this subsection provide the basis for
proving their well-posedness in Section~\ref{section:exuni}.

Finally, Subsection~\ref{subsec:weakito} is devoted to
an It\^o isometry for the weak stochastic integral driven
by a L\'evy process $L$.

\subsection{Tensor product spaces}\label{subsec:tensor}

Before we formulate the first result, we have to introduce some definitions and notation:
For two Banach spaces $E_1$ and $E_2$ we denote the \emph{algebraic tensor product}, i.e.,
the tensor product of $E_1$ and $E_2$ as vector spaces, by $E_1 \otimes E_2$.
The algebraic tensor product $E_1 \otimes E_2$ consists of all finite sums of the form
\[
    \sum\limits_{k=1}^{N} \varphi_k \otimes \psi_k,
        \quad \varphi_k \in E_1, \, \psi_k\in E_2, \, k = 1,\ldots,N.
\]
There are several ways to define a norm on this space. Here we introduce three of them:
\begin{enumerate}
\item\label{def:projtensor} \emph{Projective tensor product:}
    By taking the completion of the algebraic tensor product $E_1 \otimes E_2$
    with respect to the so-called projective norm defined for $x \in E_1 \otimes E_2$ by
    \[
        \| x \|_{E_1 \xpi E_2} := \inf\left\{ \sum\limits_{k=1}^{N} \|\varphi_k\|_{E_1} \|\psi_k\|_{E_2} : x = \sum\limits_{k=1}^{N} \varphi_k \otimes \psi_k\right\},
    \]
    the projective tensor product space $E_1 \xpi E_2$ is obtained.
    We abbreviate $E^{(\pi)} := E\xpi E$, whenever $E_1=E_2=E$.
\item\label{def:injtensor} \emph{Injective tensor product:}
    The injective norm of an element $x$ in the algebraic tensor product space $E_1 \otimes E_2$ is defined as
    \[
        \| x \|_{E_1 \xeps E_2} :=
            \sup\left\{ \Bigl|\sum\limits_{k=1}^{N} f(\varphi_k) \, g(\psi_k)\Bigr|
            : f \in B_{E_1^\prime}, \, g \in B_{E_2^\prime} \right\},
    \]
    where $B_{E_1^\prime}$, $B_{E_2^\prime}$ denote the closed unit balls in the
    dual spaces $E_j^\prime := \cL(E_j; \bbR)$, $j=1,2$, and
    $\sum_{k=1}^{N} \varphi_k \otimes \psi_k$ is any representation of $x\in E_1 \otimes E_2$.
    Note that the value of the supremum is independent of the choice
    of the representation of $x$, cf.~\cite[p.~45]{ryan2002}.
    The completion of $E_1 \otimes E_2$ with respect to this norm is
    called injective tensor product space and denoted by $E_1 \xeps E_2$.
    If $E_1=E_2=E$, the abbreviation $E^{(\varepsilon)} := E\xeps E$ is used.
\item\label{def:hilberttensor} \emph{Hilbert space tensor product:}
    If $E_1$ and $E_2$ are Hilbert spaces with inner products
    $\langle\cdot,\cdot\rangle_{E_1}$ and $\langle\cdot,\cdot\rangle_{E_2}$,
    the tensor product $E_1 \xH E_2$ is defined
    as the completion of the algebraic tensor product $E_1 \otimes E_2$
    with respect to the norm induced by the inner product
    \[
        \langle x, y \rangle_{E_1 \xH E_2}
            := \sum\limits_{k=1}^{N} \sum\limits_{\ell=1}^{M}
                    \langle \varphi_k, \vartheta_\ell \rangle_{E_1} \langle \psi_k, \chi_\ell \rangle_{E_2},
    \]
    where $x = \sum\nolimits_{k=1}^{N} \varphi_k \otimes \psi_k$ and
    $y = \sum\nolimits_{\ell=1}^{M} \vartheta_\ell \otimes \chi_\ell$
    are representations of $x$,~$y \in E_1 \otimes E_2$.
    For $E_1=E_2=E$, set $E^{(2)} := E\xH E$.
\end{enumerate}

In the latter case, we obtain again a Hilbert space,
whereas the vector spaces in~\ref{def:projtensor} and~\ref{def:injtensor} are Banach spaces,
which are in general \emph{not reflexive}, cf.~\cite[Theorem~4.21]{ryan2002}.
The following lemma is an immediate consequence of the definitions above.

\begin{lemma}\label{lem:tensor} 
    For Banach spaces $E_1$, $E_2$, $F_1$, $F_2$
    and Hilbert spaces $H_1$, $H_2$, $U_1$, $U_2$
    the following hold:
    \begin{enumerate}
    \item\label{lem:tensor-i}
        For bounded linear operators $S\in\cL(E_1; F_1)$ and $T\in\cL(E_2; F_2)$
        there exists a unique bounded linear operator
        $S\xpi T : E_1 \xpi E_2 \to F_1 \xpi F_2$
        such that $(S\xpi T)(x \otimes y) = (Sx) \otimes (Ty)$ for every $x\in E_1$, $y\in E_2$
        and it holds
        \[
            \|S\xpi T\|_{\cL(E_1 \xpi E_2; F_1 \xpi F_2)} = \|S\|_{\cL(E_1; F_1)} \|T\|_{\cL(E_2; F_2)}.
        \]
    \item\label{lem:tensor-ii}
        For bounded linear operators $S\in\cL(H_1; U_1)$ and $T\in\cL(H_2; U_2)$
        there exists a unique bounded linear operator
        $S\xH T : H_1 \xH H_2 \to U_1 \xH U_2$
        such that $(S\xH T)(x \otimes y) = (Sx) \otimes (Ty)$ for every $x\in H_1$, $y\in H_2$
        and it holds
        \[
            \|S\xH  T\|_{\cL(H_1 \xH  H_2; U_1 \xH  U_2)} = \|S\|_{\cL(H_1; U_1)} \|T\|_{\cL(H_2; U_2)}.
        \]
    \item\label{lem:tensor-iii} The following chain of continuous embeddings holds:
        \[
            H_1 \xpi H_2 \hookrightarrow H_1 \xH H_2 \hookrightarrow H_1 \xeps H_2,
        \]
        where all embedding constants are equal to 1.
    \end{enumerate} 
\end{lemma}

\begin{proof}
    For~\ref{lem:tensor-i} see~\cite[Proposition 2.3]{ryan2002}.

    To see that $S\otimes T$ is a bounded mapping with respect to the
    Hilbert tensor products in~\ref{lem:tensor-ii},
    one may proceed as
    in~\cite[Section I.2.3]{dixmier1981} -- there for the case
    $H_1 = U_1$ and $H_2 = U_2$. We may write $S\otimes T$ as
    $S \otimes T = (I_{U_1} \otimes T)(S \otimes I_{H_2})$ and for
    $x\in H_1 \otimes H_2$ we can choose a representation
    $\sum\nolimits_{k=1}^{N} \varphi_k \otimes \psi_k$ of $x$,
    such that the vectors $\{\psi_k\}_{k=1}^N$ are orthonormal in $H_2$.
    Then we obtain
    \begin{align*}
        \| (S \otimes I_{H_2}) x \|_{U_1\xH H_2}^2
             &= \Bigl\| \sum\limits_{k=1}^{N} S \varphi_k \otimes \psi_k \Bigr\|_{U_1\xH H_2}^2
              = \sum\limits_{k=1}^{N} \|S \varphi_k\|_{U_1}^2 \\
             &\leq \|S \|_{\cL(H_1; U_1)}^2  \sum\limits_{k=1}^{N} \| \varphi_k\|_{H_1}^2
              = \|S \|_{\cL(H_1; U_1)}^2 \|x\|_{H_1\xH H_2}^2
    \end{align*}
    and, thus,
    \[
        \|(S \otimes I_{H_2}) x \|_{ U_1 \xH H_2} \leq \|S \|_{\cL(H_1; U_1)} \|x\|_{H_1 \xH H_2}
    \]
    for all $x \in H_1 \otimes H_2$.
    In the same way, one can prove that
    \[
        \|(I_{U_1} \otimes T) y\|_{U_1 \xH U_2} \leq \| T\|_{\cL(H_2; U_2)} \| y \|_{U_1 \xH H_2}
    \]
    for every $y\in U_1 \otimes H_2$ and conclude for $x \in H_1 \otimes H_2$
    \begin{align*}
        \| (S \otimes T)x \|_{U_1 \xH U_2}
            &\leq \| T\|_{\cL(H_2; U_2)} \| (S \otimes I_{H_2}) x \|_{U_1 \xH H_2} \\
            &\leq \| T\|_{\cL(H_2; U_2)} \| S\|_{\cL(H_1; U_1)}  \| x \|_{H_1 \xH H_2}.
    \end{align*}
    Hence, there exists a unique continuous extension $S\xH T \in \cL(H_1 \xH H_2; U_1 \xH U_2)$
    with $\|S\xH  T\|_{\cL(H_1 \xH  H_2; U_1 \xH  U_2)} = \|S\|_{\cL(H_1; U_1)} \|T\|_{\cL(H_2; U_2)}$.

    In order to prove \ref{lem:tensor-iii}, let $x\in H_1\otimes H_2$.
    Then, we estimate
    \[
        \|x\|_{H_1 \xH H_2} = \Bigl\|\sum\limits_{k=1}^{N} \varphi_k \otimes \psi_k \Bigr\|_{H_1 \xH H_2}
            \leq \sum\limits_{k=1}^{N} \left\| \varphi_k \otimes \psi_k \right\|_{H_1 \xH H_2}
             = \sum\limits_{k=1}^{N} \| \varphi_k\|_{H_1} \|\psi_k \|_{H_2}
    \]
    for any representation $\sum\nolimits_{k=1}^{N} \varphi_k \otimes \psi_k$ of $x$.
    This shows that $\|x\|_{H_1 \xH H_2} \leq \|x\|_{H_1 \xpi H_2}$ for all $x\in H_1 \otimes H_2$ and,
    thus, $H_1 \xpi H_2 \hookrightarrow H_1 \xH H_2$ with embedding constant 1.

    Furthermore, by the Riesz representation theorem,
    for $f\in B_{H_1^\prime}$ and $g\in B_{H_2^\prime}$ there exist
    $\chi_f \in B_{H_1}$ and $\chi_g \in B_{H_2}$ such that
    $\langle \chi_f, \varphi\rangle_{H_1} = f(\varphi)$,
    $\langle \chi_g, \psi\rangle_{H_2} = g(\psi)$ for all $\varphi\in H_1$, $\psi\in H_2$.
    This yields
    {\allowdisplaybreaks
    \begin{align*}
        \Bigl| \sum\limits_{k=1}^{N} f(\varphi_k) \, g(\psi_k) \Bigr|^2
            &= \Bigl| \sum\limits_{k=1}^{N} \langle \chi_f, \varphi_k\rangle_{H_1} \langle \chi_g, \psi_k\rangle_{H_2} \Bigr|^2 \\
            &= \sum\limits_{k=1}^{N} \sum\limits_{\ell=1}^{N} \langle \chi_f, \varphi_k\rangle_{H_1} \langle \chi_g, \psi_k\rangle_{H_2}
                                                        \langle \chi_f, \varphi_\ell\rangle_{H_1} \langle \chi_g, \psi_\ell\rangle_{H_2} \\
            &= \sum\limits_{k=1}^{N} \sum\limits_{\ell=1}^{N} \langle \langle \chi_f, \varphi_\ell\rangle_{H_1} \chi_f, \varphi_k\rangle_{H_1}
                                                            \langle \langle \chi_g, \psi_\ell\rangle_{H_2} \chi_g, \psi_k\rangle_{H_2} \\
            &= \sum\limits_{k=1}^{N} \sum\limits_{\ell=1}^{N} \langle P_{\chi_f} \varphi_\ell, \varphi_k\rangle_{H_1} \|\chi_f \|_{H_1}^2
                                                            \langle P_{\chi_g} \psi_\ell, \psi_k\rangle_{H_2} \|\chi_g\|_{H_2}^2,
    \end{align*}
    }
    where
    $P_{\chi_f}$ and $P_{\chi_g}$
    denote the orthogonal projections on the subspaces
    $\operatorname{span}\{\chi_f\} := \{\alpha\,\chi_f : \alpha\in\bbR\} \subset H_1$ and
    $\operatorname{span}\{\chi_g\} := \{\alpha\,\chi_g : \alpha\in\bbR\} \subset H_2$, i.e.,
    \[
       P_{\chi_f} \varphi := \frac{\langle \chi_f, \varphi\rangle_{H_1}}{\|\chi_f\|_{H_1}^2} \, \chi_f, \quad \varphi\in H_1, \qquad
       P_{\chi_g} \psi    := \frac{\langle \chi_g, \psi   \rangle_{H_2}}{\|\chi_g\|_{H_2}^2} \, \chi_g, \quad \psi\in H_2.
    \]
    By using the properties of orthogonal projections we estimate
    \begin{align*}
        \Bigl| \sum\limits_{k=1}^{N} & f(\varphi_k) \, g(\psi_k) \Bigr|^2
             = \|\chi_f \|_{H_1}^2 \|\chi_g\|_{H_2}^2
               \sum\limits_{k=1}^{N} \sum\limits_{\ell=1}^{N} \langle P_{\chi_f} \varphi_\ell, P_{\chi_f} \varphi_k\rangle_{H_1}
                                                            \langle P_{\chi_g} \psi_\ell, P_{\chi_g} \psi_k\rangle_{H_2} \\
            &= \|\chi_f \|_{H_1}^2 \|\chi_g\|_{H_2}^2
               \sum\limits_{k=1}^{N} \sum\limits_{\ell=1}^{N}
                            \langle P_{\chi_f} \varphi_\ell \otimes P_{\chi_g} \psi_\ell, P_{\chi_f} \varphi_k \otimes P_{\chi_g} \psi_k \rangle_{H_1 \xH H_2} \\
            &= \|\chi_f \|_{H_1}^2 \|\chi_g\|_{H_2}^2
               \Bigl\langle \sum\limits_{\ell=1}^{N} P_{\chi_f} \varphi_\ell \otimes P_{\chi_g} \psi_\ell,
                            \sum\limits_{k=1}^{N} P_{\chi_f} \varphi_k \otimes P_{\chi_g} \psi_k \Bigr\rangle_{H_1 \xH H_2} \\
            &= \|\chi_f \|_{H_1}^2 \|\chi_g\|_{H_2}^2
               \Bigl\| \sum\limits_{k=1}^{N} P_{\chi_f} \varphi_k \otimes P_{\chi_g} \psi_k \Bigr\|_{H_1 \xH H_2}^2 \\
            &= \|\chi_f \|_{H_1}^2 \|\chi_g\|_{H_2}^2
               \Bigl\| (P_{\chi_f} \xH P_{\chi_g}) \sum\limits_{k=1}^{N} \varphi_k \otimes \psi_k \Bigr\|_{H_1 \xH H_2}^2,
    \end{align*}
    where $P_{\chi_f} \xH P_{\chi_g}$ denotes the extension of $P_{\chi_f} \otimes P_{\chi_g}$ to $H_1 \xH H_2$,
    which has been introduced in Lemma~\ref{lem:tensor}~\ref{lem:tensor-ii}. This lemma and
    $\|\chi_f\|_{H_1} \leq 1$, $\|\chi_g\|_{H_2} \leq 1$ yield
    \begin{align*}
        \Bigl| \sum\limits_{k=1}^{N} f(\varphi_k) \, g(\psi_k) \Bigr|^2
            &\leq \| P_{\chi_f} \xH P_{\chi_g} \|_{\cL(H_1 \xH H_2; H_1 \xH H_2)}^2 \Bigl\| \sum\limits_{k=1}^{N} \varphi_k \otimes \psi_k \Bigr\|_{H_1 \xH H_2}^2 \\
            &= \| P_{\chi_f}\|_{\cL(H_1;H_1)}^2 \| P_{\chi_g} \|_{\cL(H_2; H_2)}^2 \| x \|_{H_1 \xH H_2}^2
             = \| x \|_{H_1 \xH H_2}^2
    \end{align*}
    for any representation $\sum\nolimits_{k=1}^{N} \varphi_k \otimes \psi_k$
    of $x \in H_1 \otimes H_2$.
    Since $f\in B_{H_1^\prime}$ and $g\in B_{H_2^\prime}$ were arbitrarily chosen we obtain
    \[
        \|x\|_{H_1 \xeps H_2}
                = \sup\left\{ \Bigl|\sum\limits_{k=1}^{N} f(\varphi_k) \, g(\psi_k)\Bigr|
                            : f \in B_{H_1^\prime}, \, g \in B_{H_2^\prime} \right\}
                \leq  \| x\|_{H_1 \xH H_2}.
    \]
    This yields $H_1 \xH H_2 \hookrightarrow H_1 \xeps H_2$ with embedding constant 1
    and completes the proof.
\end{proof}

For our purpose -- formulating variational problems on tensor product spaces
for the second moment and the covariance of the mild solution
to the stochastic partial differential equation -- the following
result on the dual space of the injective tensor product of separable Hilbert spaces
will be important.

\begin{lemma}\label{lem:dualequal}
    Let $H_1$ and $H_2$ be separable Hilbert spaces. Then the dual space
    of the injective tensor product space is isometrically isomorphic to the projective tensor product
    of the dual spaces, i.e., $(H_1 \xeps H_2)^\prime \cong H_1^\prime \xpi H_2^\prime$. 
\end{lemma}

\begin{proof}
    The proof can be extracted from \cite{ryan2002} as follows:
    The dual space of the injective tensor product space can be identified with the Banach space
    of integral bilinear forms on~$H_1 \times H_2$ by~\cite[Proposition~3.14]{ryan2002}.
    In addition, since $H_1$ and $H_2$ are separable Hilbert spaces,
    the dual spaces $H_1^\prime$ and $H_2^\prime$ have the so-called approximation property,
    which implies that the projective tensor product of them can be identified with
    the Banach space of nuclear bilinear forms on~$H_1 \times H_2$ by~\cite[Corollary~4.8~(b)]{ryan2002}.
    In general, the space of nuclear bilinear forms
    is only a subspace of the space of integral bilinear forms. 
    Since we assume that $H_1$ and $H_2$ are separable Hilbert spaces,
    they have monotone shrinking Schauder bases
    and this fact implies that every integral bilinear form
    on $H_1 \times H_2$ is nuclear and the integral and nuclear norms coincide, cf.~\cite[Corollary~4.29]{ryan2002}.
    Hence, the spaces $(H_1 \xeps H_2)^\prime$ and $H_1^\prime \xpi H_2^\prime$
    are isometrically isomorphic. 
\end{proof}

\subsection{The covariance kernel and the multiplicative noise}\label{subsec:covkern}

For a $U$-valued L\'evy process $L$ with covariance operator $Q$ as considered in Section~\ref{section:spde},
we define the covariance kernel $q\in U^{(2)}$ as the unique element in the tensor space $U^{(2)}$ satisfying
\begin{equation}\label{eq:qdef}
    \utip{q}{x \otimes y} = \uip{Q x}{y}
\end{equation}
for all $x$, $y \in U$.
Note that for an orthonormal eigenbasis $(e_n)_{n\in\bbN}\subset U$ of $Q$
with corresponding eigenvalues $(\gamma_n)_{n\in\bbN}$ we may expand 
\begin{equation}\label{eq:qexp}
    q = \sum\limits_{n\in\bbN} \sum\limits_{m\in\bbN} \utip{q}{e_n\otimes e_m} (e_n\otimes e_m)
      = \sum\limits_{m\in\bbN} \gamma_m (e_m\otimes e_m)
\end{equation}
with convergence of the series in $U^{(2)}$,
since $(e_n \otimes e_m)_{n,m\in\bbN}$ is an orthonormal basis of $U^{(2)}$ and
$\utip{q}{e_n\otimes e_m} = \gamma_m \delta_{nm}$, where $\delta_{nm}$ denotes the Kronecker delta.
In addition, we obtain convergence of the series also with respect to $U^{(\pi)}$, which is shown
in the following lemma.

\begin{lemma}\label{lem:qupiconv}
    The series in \eqref{eq:qexp} converges in $U^{(\pi)}$, i.e.,
    \[
        \lim\limits_{M\to \infty} \Bigl\| q - \sum\limits_{m=1}^{M} \gamma_m (e_m \otimes e_m) \Bigr\|_{U^{(\pi)}} = 0.
    \]
\end{lemma}

\begin{proof}
    For $M\in\bbN$ define
    \begin{equation}\label{eq:qMdef}
        q_M := \sum\limits_{m=1}^{M} \gamma_m (e_m\otimes e_m).
    \end{equation}
    The trace class property of $Q$ implies that $\sum\nolimits_{m\in\bbN} \gamma_m < +\infty$. Hence, for
    any $\epsilon > 0$ there exists $N_0 \in \bbN$ such that $\sum\nolimits_{m=M+1}^{L} \gamma_m < \epsilon$ for all $L>M\geq N_0$
    and $(q_M)_{M\in\bbN}$ is a Cauchy sequence in $U^{(\pi)}$, since for any $L>M\geq N_0$ we obtain
    \[
        \| q_L - q_M \|_{U^{(\pi)}} = \Bigl\| \sum\limits_{m=M+1}^{L} \gamma_m (e_m \otimes e_m) \Bigr\|_{U^{(\pi)}}
                    \leq \sum\limits_{m=M+1}^{L} \gamma_m < \epsilon.
    \]
    The completeness of the space $U^{(\pi)}$ implies the existence of $q_* \in U^{(\pi)}$ such that
    $\lim\limits_{M\to\infty} \|q_M - q_* \|_{U^{(\pi)}} = 0$. The convergence $\lim\limits_{M\to\infty} q_M = q$
    in $U^{(2)}$ and the continuous embedding $U^{(\pi)} \hookrightarrow U^{(2)}$,
    cf.~Lemma~\ref{lem:tensor}~\ref{lem:tensor-iii}, yield $q = q_* \in U^{(\pi)}$.
\end{proof}

The bilinear form and the right-hand side appearing in the deterministic variational problems
in Sections~\ref{section:moment}~and~\ref{section:covariance}, contain several
terms depending on the operators $G_1$ and $G_2$ as well as on the
kernel $q$ that is associated with the covariance operator $Q$ via~\eqref{eq:qdef}.
To verify that they are well-defined we introduce the following Bochner spaces
as well as their inner products
\begin{align*}
    \hspace*{2cm} \cW &:=  L^2(\bbT; H), &&& \langle u_1, u_2 \rangle_\cW &:= \int_0^T \hip{u_1(t)}{u_2(t)} \, \rd t, \hspace*{2cm} \\
    \hspace*{2cm} \cX &:=  L^2(\bbT; V), &&& \langle v_1, v_2 \rangle_\cX &:= \int_0^T \langle v_1(t), v_2(t)\rangle_V \, \rd t \hspace*{2cm}
\end{align*}
and derive the results of the two lemmas below.
\begin{lemma}\label{lem:ggupi}
    For operators $G_1$ and $G_2$ satisfying Assumption~\ref{ass:spde}~\ref{ass:spde-v} the following properties hold:
    \begin{enumerate}
    \item\label{lem:ggupi-i}
        The linear operator $G_1 \otimes G_1 \colon U \otimes U \to \cL(\cX;\cW) \otimes \cL(\cX;\cW)$,
        \[
            \sum\limits_{\ell=1}^{M} \varphi_\ell^1 \otimes \varphi_\ell^2 \mapsto \sum\limits_{\ell=1}^{M} G_1(\cdot) \varphi_\ell^1 \otimes G_1(\cdot) \varphi_\ell^2
        \]
        admits a unique extension $G_1 \xpi G_1 \in \cL(U^{(\pi)}; \cL(\cX^{(\pi)};\cW^{(\pi)}))$.
    \item\label{lem:ggupi-ii}
        The linear operators $G_1 \otimes G_2 \colon U \otimes U \to \cL(\cX;\cW) \otimes H$
        and $G_2 \otimes G_1 \colon U \otimes U \to H \otimes \cL(\cX;\cW)$,
        \[
            \sum\limits_{\ell=1}^{M} \varphi_\ell^1 \otimes \varphi_\ell^2 \mapsto \sum\limits_{\ell=1}^{M} G_1(\cdot) \varphi_\ell^1 \otimes G_2 \varphi_\ell^2, \quad
            \sum\limits_{\ell=1}^{M} \varphi_\ell^1 \otimes \varphi_\ell^2 \mapsto \sum\limits_{\ell=1}^{M} G_2 \varphi_\ell^1 \otimes G_1(\cdot) \varphi_\ell^2
        \]
        admit unique extensions $G_1 \xpi G_2 \in \cL(U^{(\pi)}; \cL(\cX;\cW \xpi H))$
        and $G_2 \xpi G_1 \in \cL(U^{(\pi)}; \cL(\cX; H \xpi\cW))$.
    \item\label{lem:ggupi-iii}
        The linear operator $G_2 \otimes G_2 \colon U \otimes U \to H \otimes H$,
        \[
            \sum\limits_{\ell=1}^{M} \varphi_\ell^1 \otimes \varphi_\ell^2 \mapsto \sum\limits_{\ell=1}^{M} G_2 \varphi_\ell^1 \otimes G_2 \varphi_\ell^2
        \]
        admits a unique extension $G_2 \xpi G_2 \in \cL(U^{(\pi)}; H^{(\pi)})$.
    \end{enumerate}
\end{lemma}

\begin{proof}
    We first note that $G_1 \in \cL(V; \cL(U;H))$ implies that
    $G_1$ can be identified with an element in $\cL(U;\cL(\cX;\cW))$,
    because for any $\varphi \in U$ we estimate
    \begin{align*}
        \|G_1&(\cdot)\varphi\|_{\cL(\cX;\cW)} = \sup\limits_{\substack{u \in \cX \\ \xnorm{u} = 1}} \| G_1(u)\varphi \|_\cW
                  = \sup\limits_{\substack{u \in \cX \\ \xnorm{u} = 1}} \left( \int_0^T \| G_1(u(t)) \varphi \|_H^2 \, \rd t \right)^{\frac{1}{2}} \\
                &\leq \|\varphi \|_U \sup\limits_{\substack{u \in \cX \\ \xnorm{u} = 1}} \left( \int_0^T \| G_1(u(t)) \|_{\cL(U;H)}^2 \, \rd t \right)^{\frac{1}{2}} \\
                &\leq \|\varphi \|_U \sup\limits_{\substack{u \in \cX \\ \xnorm{u} = 1}} \left( \int_0^T \| G_1\|_{\cL(V;\cL(U;H))}^2 \|u(t) \|_{V}^2 \, \rd t \right)^{\frac{1}{2}}
                 \leq \| G_1 \|_{\cL(V; \cL(U;H))} \|\varphi\|_U.
    \end{align*}
    This inequality shows that
    \[
        G_1 \in \cL(U; \cL(\cX; \cW)), \quad \| G_1\|_{\cL(U; \cL(\cX; \cW))} \leq \| G_1 \|_{\cL(V; \cL(U;H))}.
    \]

    In order to prove \ref{lem:ggupi-i} note that by Lemma~\ref{lem:tensor}~\ref{lem:tensor-i} for two vectors $\varphi^1$, $\varphi^2 \in U$
    there exists a unique operator $G_1(\cdot)\varphi^1 \xpi G_1(\cdot)\varphi^2 \colon \cX^{(\pi)} \to \cW^{(\pi)}$
    satisfying
    \[
        \left(G_1(\cdot)\varphi^1 \xpi G_1(\cdot)\varphi^2\right)(u) = \sum\limits_{k=1}^{N} G_1(u_k^1)\varphi^1 \otimes G_1(u_k^2)\varphi^2
    \]
    for any representation $\sum\nolimits_{k=1}^{N} u_k^1 \otimes u_k^2$
    of $u\in\cX\otimes\cX$. This operator is bounded because
    \[
        \| G_1(\cdot)\varphi^1 \xpi G_1(\cdot)\varphi^2\|_{\cL(\cX^{(\pi)}; \cW^{(\pi)})} = \| G_1(\cdot)\varphi^1\|_{\cL(\cX;\cW)} \| G_1(\cdot)\varphi^2\|_{\cL(\cX;\cW)}.
    \]
    In addition, for a representation $\sum\nolimits_{\ell=1}^{M} \varphi_\ell^1 \otimes \varphi_\ell^2$
    of $\varphi \in U\otimes U$ we may extend
    \[
        (G_1(\cdot) \otimes G_1(\cdot)) \varphi = \sum\limits_{\ell=1}^{M} G_1(\cdot)\varphi_\ell^1 \otimes G_1(\cdot)\varphi_\ell^2 \colon \cX \otimes \cX \to \cW \otimes \cW
    \]
    to a bounded linear operator $(G_1(\cdot) \xpi G_1(\cdot)) \varphi \in\cL(\cX^{(\pi)}; \cW^{(\pi)})$, since
    \begin{align*}
        \|(G_1&(\cdot) \otimes G_1(\cdot)) \varphi\|_{\cL(\cX^{(\pi)}; \cW^{(\pi)})}
                 \leq \sum\limits_{\ell=1}^{M} \| G_1(\cdot)\varphi_\ell^1 \otimes G_1(\cdot)\varphi_\ell^2 \|_{\cL(\cX^{(\pi)}; \cW^{(\pi)})} \\
                &= \sum\limits_{\ell=1}^{M} \| G_1(\cdot)\varphi_\ell^1 \|_{\cL(\cX;\cW)} \|G_1(\cdot)\varphi_\ell^2 \|_{\cL(\cX; \cW)}
                 \leq  \| G_1\|_{\cL(V; \cL(U;H))}^2 \sum\limits_{\ell=1}^{M} \|\varphi^1\|_U \|\varphi^2\|_U
    \end{align*}
    by the observations above. Therefore, $(G_1(\cdot) \otimes G_1(\cdot)) \varphi \in \cL(\cX^{(\pi)}; \cW^{(\pi)})$
    for all $\varphi\in U\otimes U$ with
    \[
        \| (G_1(\cdot) \otimes G_1(\cdot))\varphi \|_{\cL(\cX^{(\pi)}; \cW^{(\pi)})}
            \leq \| G_1 \|_{\cL(V; \cL(U;H))}^2 \|\varphi\|_{U^{(\pi)}}.
    \]
    This estimate shows that 
    $G_1 \otimes G_1 \colon U\otimes U \to \cL(\cX;\cW) \otimes \cL(\cX;\cW)$ admits a unique continuous extension
    to an operator $G_1 \xpi G_1 \in \cL(U^{(\pi)}; \cL(\cX^{(\pi)};\cW^{(\pi)}))$.

    For part~\ref{lem:ggupi-ii}, let $\sum\nolimits_{\ell=1}^{M} \varphi_\ell^1 \otimes \varphi_\ell^2$ be again a representation of $\varphi \in U\otimes U$.
    Then, for $u\in\cX$ we calculate
    \begin{align*}
        \Bigl\| \sum\limits_{\ell=1}^{M} G_1(u) \varphi_\ell^1 &\otimes G_2 \varphi_\ell^2 \Bigr\|_{\cW \xpi H}
             \leq \sum\limits_{\ell=1}^{M} \| G_1(u) \varphi_\ell^1 \|_\cW \|G_2 \varphi_\ell^2 \|_H \\
            &\leq \sum\limits_{\ell=1}^{M} \| G_1(\cdot)\varphi_\ell^1  \|_{\cL(\cX; \cW)} \xnorm{u} \|G_2 \|_{\cL(U;H)} \|\varphi_\ell^2 \|_U \\
            &\leq \| G_1\|_{\cL(V;\cL(U;H))} \|G_2 \|_{\cL(U;H)} \xnorm{u} \sum\limits_{\ell=1}^{M} \|\varphi_\ell^1 \|_U \|\varphi_\ell^2 \|_U.
    \end{align*}
    This calculation implies that $(G_1(\cdot)\otimes G_2)\varphi \in \cL(\cX; \cW \xpi H)$
    for any $\varphi \in U\otimes U$ with
    \[
        \| (G_1(\cdot)\otimes G_2)\varphi \|_{\cL(\cX; \cW \xpi H)}
            \leq \| G_1\|_{\cL(V;\cL(U;H))} \|G_2 \|_{\cL(U;H)} \|\varphi\|_{U^{(\pi)}},
    \]
    and that there exists a unique extension $G_1 \xpi G_2 \in \cL(U^{(\pi)}; \cL(\cX; \cW \xpi H))$.
    It is obvious that the same argumentation yields existence and uniqueness of an extension
    $G_2 \xpi G_1 \in \cL(U^{(\pi)}; \cL(\cX; H \xpi \cW))$ of $G_2 \otimes G_1$.

    Assertion~\ref{lem:ggupi-iii} follows immediately, since $G_2 \in \cL(U;H)$ implies the existence of $G_2 \xpi G_2 \in \cL(U^{(\pi)}; H^{(\pi)})$
    by Lemma~\ref{lem:tensor}~\ref{lem:tensor-i}.
\end{proof}

\begin{lemma}\label{lem:ggqvv}
    Define $q\in U^{(2)}$ as in~\eqref{eq:qdef} and let $G_1$ and $G_2$ satisfy
    Assumption~\ref{ass:spde}~\ref{ass:spde-v}. Then,
    \begin{enumerate}
    \item\label{lem:ggqvv-i}
        $(G_1\otimes G_1)(\cdot)q \colon \cX^{(\pi)} \to \cW^{(\pi)}$ is bounded and
        \begin{equation}\label{eq:ggqvv}
            \| (G_1\otimes G_1)(\cdot)q \|_{\cL( \cX^{(\pi)}; \cW^{(\pi)} )}
                \leq \|G_1 \|_{\cL(V; \lhs(\cH; H))}^2;
        \end{equation}
    \item\label{lem:ggqvv-ii}
        $(G_1(\cdot) \otimes G_2)q \in \cL(\cX; \cW \xpi H)$ and
        $(G_2 \otimes G_1(\cdot))q \in \cL(\cX; H \xpi \cW)$;
    \item\label{lem:ggqvv-iii}
        $(G_2 \otimes G_2) q \in H^{(\pi)}$.
    \end{enumerate}
\end{lemma}

\begin{proof}
    The results
    $(G_1 \otimes G_1)(\cdot)q \in \cL(\cX^{(\pi)}; \cW^{(\pi)} )$,
    $(G_1(\cdot) \otimes G_2)q \in \cL(\cX; \cW \xpi H)$,
    $(G_2 \otimes G_1(\cdot))q \in \cL(\cX; H \xpi \cW)$
    and $(G_2 \otimes G_2) q \in H^{(\pi)}$ are immediate consequences of Lemma~\ref{lem:ggupi},
    since $q\in U^{(\pi)}$ by Lemma~\ref{lem:qupiconv}.

    In order to prove the bound in~\eqref{eq:ggqvv}, let $M\in\bbN$
    and define $q_M \in U\otimes U$ as in~\eqref{eq:qMdef}.
    Set $f_m := \sqrt{\gamma_m} \, e_m$, $m\in\bbN$,
    and let $\sum\nolimits_{k=1}^{N} u_k^1 \otimes u_k^2$
    be a representation of $u \in \cX \otimes \cX$. 
    Then,
    \begin{align*}
        \| (G_1 \xpi G_1)(u) q_M \|_{\cW^{(\pi)}}
            &\leq \sum\limits_{k=1}^{N} \sum\limits_{m=1}^{M} \| G_1(u_k^1) f_m \|_\cW \| G_1(u_k^2) f_m \|_\cW \\
            &\leq \sum\limits_{k=1}^{N} \left( \sum\limits_{m=1}^{M} \| G_1(u_k^1) f_m \|_\cW^2 \right)^{\frac{1}{2}}
                                        \left( \sum\limits_{m=1}^{M} \| G_1(u_k^2) f_m \|_\cW^2 \right)^{\frac{1}{2}} \\
            &\leq \| G_1\|_{\cL(V;\lhs(\cH;H))}^2 \sum\limits_{k=1}^{N} \|u_k^1 \|_\cX \| u_k^2 \|_\cX,
    \end{align*}
    since for $v \in \cX$ we obtain
    \begin{align*}
        \sum\limits_{m=1}^{M} \| G_1(v) f_m \|_\cW^2
            &= \int_0^T \sum\limits_{m=1}^{M} \| G_1(v(t)) f_m \|_H^2 \, \rd t
             \leq  \int_0^T \| G_1(v(t)) \|_{\lhs(\cH;H)}^2 \, \rd t,
        \intertext{where the last inequality follows from the fact that the set $\{f_j : j\in\bbN, \, \gamma_j \neq 0\}$ forms
        an orthonormal basis of $\cH$. Therefore, }
        \sum\limits_{m=1}^{M} \| G_1(v) f_m \|_\cW^2
            &\leq \| G_1\|_{\cL(V;\lhs(\cH;H))}^2 \int_0^T \| v(t)\|_{V}^2 \, \rd t = \| G_1\|_{\cL(V;\lhs(\cH;H))}^2 \xnorm{v}^2
    \end{align*}
    and, hence, $(G_1 \xpi G_1)(\cdot) q_M \in \cL(\cX^{(\pi)}; \cW^{(\pi)})$ for all $M\in\bbN$ with
    \[
        \|(G_1 \otimes G_1)(\cdot) q_M \|_{\cL(\cX^{(\pi)}; \cW^{(\pi)})} \leq \| G_1\|_{\cL(V;\lhs(\cH;H))}^2.
    \]
    The bound for $(G_1 \xpi G_1)(\cdot) q$ in \eqref{eq:ggqvv} follows from Lemmas~\ref{lem:qupiconv} and~\ref{lem:ggupi}~\ref{lem:ggupi-i}, since
    $\lim\nolimits_{M\to\infty} q_M = q$ in $U^{(\pi)}$ and $G_1 \xpi G_1 \in \cL(U^{(\pi)}; \cL(\cX^{(\pi)}; \cW^{(\pi)}))$.
\end{proof}

\subsection{The diagonal trace operator}\label{subsec:diagonal}

We introduce the spaces
$H^1_{0,\{T\}}(\bbT; V^*) := \left\{v \in H^1(\bbT; V^*) : v(T)=0\right\}$
as well as $\cY := L^2(\bbT; V) \cap H^1_{0,\{T\}}(\bbT; V^*)$. $\cY$ is a Hilbert space with respect
to the inner product
\[
    \langle v_1, v_2 \rangle_\cY := \langle v_1, v_2 \rangle_{L^2(\bbT; V)}
                                  + \langle \partial_t v_1, \partial_t v_2 \rangle_{L^2(\bbT; V^*)},
    \quad v_1, v_2 \in \cY.
\]
Moreover, we obtain the following two continuous embeddings.
\begin{lemma}\label{lem:C0Y}
    It holds that $\cY \hookrightarrow C^0(\bbT; H)$ with embedding constant $C\leq 1$, i.e.,
    $\sup\limits_{s\in\bbT} \hnorm{v(s)} \leq \ynorm{v}$ for every $v \in \cY$.
\end{lemma}

\begin{proof}
    For every $v \in \cY =  L^2(\bbT; V) \cap H_{0, \{T\}}^1(\bbT; V^*)$ we have the relation
    \[
        \hnorm{v(r)}^2 - \hnorm{v(s)}^2 = \int_s^r 2 \, \vvdp{\partial_t v(t)}{v(t)} \, \rd t ,
        \quad r,s \in \bbT, \, r>s,
    \]
    cf.~\cite[§XVIII.1, Theorem~2]{dautray1999}.
    Choosing $r=T$ and observing that $v(T)=0$ leads to
    \[
        \hnorm{v(s)}^2 \leq 2 \, \| \partial_t v \|_{L^2(\bbT; V^*)} \| v\|_{L^2(\bbT; V)}
                       \leq \| \partial_t v \|_{L^2(\bbT; V^*)}^2 + \| v\|_{L^2(\bbT; V)}^2 = \ynorm {v}^2.
                       \qedhere
    \]
\end{proof}

\begin{lemma}\label{lem:C0Yeps}
    The injective tensor product space satisfies
    $\cY^{(\varepsilon)} \hookrightarrow C^0(\bbT; H)^{(\varepsilon)}$
    with embedding constant $C\leq 1$.
\end{lemma}

\begin{proof}
    The continuous embedding of Lemma~\ref{lem:C0Y}
    implies that $\| g \|_{\cY^\prime} \leq \| g \|_{C^0(\bbT; H)^\prime}$ for all $g\in C^0(\bbT; H)^\prime$.
    Therefore, the unit balls of the dual spaces satisfy $B_{C^0(\bbT; H)^\prime} \subset B_{\cY^\prime}$ and
    the embedding of the injective tensor product spaces follows,
    since for $\sum_{k=1}^{N} v_k^1 \otimes v_k^2 \in \cY \otimes \cY$
    we obtain 
    \begin{align*}
        \Bigl \| \sum\limits_{k=1}^{N} v_k^1 &\otimes v_k^2 \Bigr\|_{C^0(\bbT; H)^{(\varepsilon)}}
             = \sup\left\{ \Bigl| \sum\limits_{k=1}^{N} f\bigl(v_k^1\bigr) \, g\bigl(v_k^2\bigr) \Bigr| :
                            f,g \in B_{C^0(\bbT; H)^\prime} \right\} \\
             &\leq \sup\left\{ \Bigl| \sum\limits_{k=1}^{N} f\bigl(v_k^1\bigr) \, g\bigl(v_k^2\bigr) \Bigr| :
                            f,g \in B_{\cY^\prime} \right\}
             = \Bigl \| \sum\limits_{k=1}^{N} v_k^1 \otimes v_k^2 \Bigr\|_{\cY^{(\varepsilon)}}. \qedhere
    \end{align*}
\end{proof}

In the deterministic equations satisfied by the second moment and the covariance, 
an operator associated with the
\emph{diagonal trace} will play an important role.
For $u \in \cW\otimes\cW$, $v\in\cY \otimes \cY$ and representations
$\sum\nolimits_{k=1}^{N} u_k^1 \otimes u_k^2$ and $\sum\nolimits_{\ell=1}^{M} v_\ell^1 \otimes v_\ell^2$
of $u$ and $v$, respectively, we define
\begin{equation}\label{eq:Tdelta}
    T_\delta(u)v := \sum\limits_{k=1}^{N} \sum\limits_{\ell=1}^{M} \int_0^T \hip{u_k^1(t)}{v_\ell^1(t)} \hip{u_k^2(t)}{v_\ell^2(t)} \, \rd t.
\end{equation}
In addition, for $\widetilde{u} \in \cW\otimes H$ and $\hat{u} \in H\otimes\cW$
with representations $\sum\nolimits_{k=1}^{N} u_k \otimes \varphi_k$,
and $\sum\nolimits_{k=1}^{N} \varphi_k \otimes u_k$, $u_k\in\cW$, $\varphi_k\in H$, respectively,
as well as $\varphi \in H\otimes H$ with representation
$\sum\nolimits_{k=1}^{N} \varphi_k^1 \otimes \varphi_k^2$ we define $T_\delta$ accordingly,
\begin{align*}
    T_\delta(\widetilde{u})v &:= \sum\limits_{k=1}^{N} \sum\limits_{\ell=1}^{M} \int_0^T \hip{u_k(t)}{v_\ell^1(t)} \hip{\varphi_k}{v_\ell^2(t)} \, \rd t, \\
    T_\delta(\hat{u})v       &:= \sum\limits_{k=1}^{N} \sum\limits_{\ell=1}^{M} \int_0^T \hip{\varphi_k}{v_\ell^1(t)} \hip{u_k(t)}{v_\ell^2(t)} \, \rd t, \\
    T_\delta(\varphi)v       &:= \sum\limits_{k=1}^{N} \sum\limits_{\ell=1}^{M} \int_0^T \hip{\varphi_k^1}{v_\ell^1(t)} \hip{\varphi_k^2}{v_\ell^2(t)} \, \rd t.
\end{align*}
With these definitions, $T_\delta$ admits unique extensions to bounded linear operators
mapping from the projective tensor spaces
$\cW \xpi \cW$, $\cW \xpi H$, $H \xpi \cW$, and $H\xpi H$, respectively, to the dual space
$\cY^{(\varepsilon)\prime} = \cL(\cY^{(\varepsilon)};\bbR)$ of the injective tensor space
$\cY \xeps \cY$ as we prove in the following proposition.

\begin{proposition}\label{prop:deltaggq}
    The operator $T_\delta \colon (\cW\otimes\cW)\times (\cY\otimes\cY) \to \bbR$ defined in~\eqref{eq:Tdelta}
    admits a unique extension to a bounded linear operator $T_\delta\in\cL(\cW^{(\pi)}; \cY^{(\varepsilon)\prime})$ with
    $\| T_\delta\|_{\cL(\cW^{(\pi)}; \cY^{(\varepsilon)\prime})} \leq 1$.
    Furthermore, $T_\delta$ as an operator acting on $\cW\otimes H$, $H \otimes\cW$, and $H\otimes H$
    admits unique extensions to $T_\delta\in\cL(\cW \xpi H; \cY^{(\varepsilon)\prime})$,
    $T_\delta\in\cL(H \xpi \cW; \cY^{(\varepsilon)\prime})$, and
    $T_\delta\in\cL(H^{(\pi)}; \cY^{(\varepsilon)\prime})$, respectively.
\end{proposition}

\begin{proof}
    Let $u\in \cW\otimes\cW$ and $v \in\cY\otimes\cY$ with representations
    $u = \sum\nolimits_{k=1}^{N} u_k^1 \otimes u_k^2$ and
    $v = \sum\nolimits_{\ell=1}^{M} v_\ell^1 \otimes v_\ell^2$ be given. 
    Then,
    \begin{align*}
        |T_\delta(u)v | &= \Bigl| \sum\limits_{k=1}^{N} \sum\limits_{\ell=1}^{M} \int_0^T \hip{u_k^1(t)}{v_\ell^1(t)} \hip{u_k^2(t)}{v_\ell^2(t)} \, \rd t \Bigr| \\
            &\leq \sum\limits_{k=1}^{N} \int_0^T \hnorm{u_k^1(t)} \hnorm{u_k^2(t)}
                \Bigl| \sum\limits_{\ell=1}^{M} \frac{\hip{u_k^1(t)}{v_\ell^1(t)}}{\hnorm{u_k^1(t)}}
                                                \frac{\hip{u_k^2(t)}{v_\ell^2(t)}}{\hnorm{u_k^2(t)}} \Bigr| \,\rd t \\
            &\leq \sum\limits_{k=1}^{N} \int_0^T \hnorm{u_k^1(t)} \hnorm{u_k^2(t)}
                                \Bigl\| \sum\limits_{\ell=1}^{M} v_\ell^1(t) \otimes v_\ell^2(t) \Bigr\|_{H^{(\varepsilon)}} \,\rd t,
    \intertext{since $\tfrac{\langle\varphi, \cdot\rangle_H}{\|\varphi\|_H} \in B_{H^\prime}$ for $\varphi\in H\setminus\{0\}$. Therefore,}
        |T_\delta(u)v |
            &\leq \sup\limits_{t\in\bbT} \Bigl\| \sum\limits_{\ell=1}^{M} v_\ell^1(t) \otimes v_\ell^2(t) \Bigr\|_{H^{(\varepsilon)}}
                                \sum\limits_{k=1}^{N} \int_0^T \hnorm{u_k^1(t)} \hnorm{u_k^2(t)} \, \rd t \\
            &\leq \sup\limits_{t\in\bbT} \sup\limits_{f,g\in B_{H^\prime}} \Bigl| \sum\limits_{\ell=1}^{M} f\bigl(v_\ell^1(t)\bigr) \, g\bigl(v_\ell^2(t)\bigr) \Bigr| \,
                                \sum\limits_{k=1}^{N} \| u_k^1 \|_\cW \|u_k^2 \|_\cW \\
            &\leq \sup\limits_{s,t\in\bbT} \sup\limits_{f,g\in B_{H^\prime}} \Bigl| \sum\limits_{\ell=1}^{M} f\bigl(\delta_s(v_\ell^1)\bigr) \,
                                                                                                             g\bigl(\delta_t(v_\ell^2)\bigr) \Bigr| \,
                                \sum\limits_{k=1}^{N} \| u_k^1 \|_\cW \|u_k^2 \|_\cW,
    \end{align*}
    where $\delta_t \colon C^0(\bbT;H) \to H$ denotes the evaluation functional in $t\in\bbT$, i.e.,
    $\delta_t(v) := v(t)$. We obtain the estimate
    \[
        |T_\delta(u)v |
             \leq \sup\limits_{\tilde{f},\tilde{g}\in B_{C^0(\bbT;H)^\prime}} \Bigl| \sum\limits_{\ell=1}^{M} f\bigl(v_\ell^1\bigr) \, g\bigl(v_\ell^2\bigr) \Bigr| \,
                                \sum\limits_{k=1}^{N} \| u_k^1 \|_\cW \|u_k^2 \|_\cW,
    \]
    because $f \circ \delta_t \in B_{C^0(\bbT; H)^\prime}$ for $f \in B_{H^\prime}$ and $t\in\bbT$. Hence,
    \[
        |T_\delta(u)v |
             \leq \| v\|_{C^0(\bbT;H)^{(\varepsilon)}}  \| u\|_{\cW^{(\pi)}}
             \leq \| v\|_{\cY^{(\varepsilon)}} \| u\|_{\cW^{(\pi)}},
    \]
    since $\cY^{(\varepsilon)} \hookrightarrow C^0(\bbT;H)^{(\varepsilon)}$
    with embedding constant 1 by Lemma~\ref{lem:C0Yeps},
    and $T_\delta$ admits a unique extension
    $T_\delta \in \cL(\cW^{(\pi)}; \cY^{(\varepsilon)\prime})$. 

    For $\widetilde{u}\in \cW\otimes H$ and $\hat{u}\in H\otimes\cW$ with representations
    $\sum\nolimits_{k=1}^{N} u_k^1 \otimes \varphi_k$ and
    $\sum\nolimits_{k=1}^{N} \varphi_k \otimes u_k^2$, respectively,
    one can prove in the same way as above that
    \[
        |T_\delta(\widetilde{u})v |
            \leq \sqrt{T} \, \| v\|_{\cY^{(\varepsilon)}} \|\widetilde{u}\|_{\cW \xpi H}, \quad
        |T_\delta(\hat{u})v |
            \leq \sqrt{T} \, \| v\|_{\cY^{(\varepsilon)}} \|\hat{u} \|_{H \xpi \cW}
    \]
    for all $v\in\cY^{(\varepsilon)}$.
    Finally, for $\varphi\in H\otimes H$ with representation
    $\sum\nolimits_{k=1}^{N} \varphi_k^1 \otimes \varphi_k^2$
    we obtain for all $v\in\cY^{(\varepsilon)}$
    \[
        |T_\delta(\varphi)v | \leq T \, \| v\|_{\cY^{(\varepsilon)}} \|\varphi \|_{H^{(\pi)}}.
    \]
    The last three estimates show that there exist unique extensions
    $T_\delta \in \cL(\cW \xpi H; \cY^{(\varepsilon)\prime})$,
    $T_\delta \in \cL(H \xpi \cW; \cY^{(\varepsilon)\prime})$, and
    $T_\delta \in \cL(H^{(\pi)}; \cY^{(\varepsilon)\prime})$
    and complete the proof. 
\end{proof}

In addition to $T_\delta$ we define the operator
$R_{t}\colon H \to \cY^\prime$ for $t\in\bbT$ by
\begin{equation}\label{eq:Rt}
    R_{t}(\varphi)v := \hip{\varphi}{v(t)}, \quad v\in\cY.
\end{equation}
The next lemma shows that we obtain a well-defined operator $R_{s,t} \in \cL(H^{(\pi)}; \cY^{(\varepsilon)\prime})$
by setting $R_{s,t} := R_s \xpi R_t$ for $s,t\in\bbT$.

\begin{lemma}\label{lem:Rst}
    The operator $R_t$ defined for $t\in\bbT$ in \eqref{eq:Rt} is bounded and satisfies
    $\|R_t\|_{\cL(H;\cY^\prime)} \leq 1$.
    Furthermore, for $s,t\in\bbT$ the operator $R_{s,t}: H\otimes H \to \cY^\prime \otimes \cY^\prime$
    defined for $\varphi\in H\otimes H$ by
    \begin{equation}\label{eq:Rst}
        R_{s,t} (\varphi) := (R_s \otimes R_t) (\varphi) = \sum\limits_{k=1}^{N} R_s(\varphi_k^1) \otimes R_t(\varphi_k^2) ,
    \end{equation}
    where $\sum\nolimits_{k=1}^{N} \varphi_k^1 \otimes \varphi_k^2$ is a representation of $\varphi\in H\otimes H$,
    admits a unique extension to a bounded linear operator $R_{s,t} \in \cL(H^{(\pi)}; \cY^{(\varepsilon)\prime})$.
\end{lemma}

\begin{proof}
    For $t\in\bbT$ and $\varphi\in H$ we calculate by using
    the Cauchy--Schwarz inequality and Lemma~\ref{lem:C0Y},
    \[
        | R_t(\varphi)v| = | \hip{\varphi}{v(t)} |
            \leq \hnorm{\varphi} \hnorm{v(t)}
            \leq \hnorm{\varphi} \| v\|_{C^0(\bbT;H)}
            \leq \hnorm{\varphi} \ynorm{v}
    \]
    for all $v\in\cY$. This proves that $R_t(\varphi) \in \cY^\prime$ for all $\varphi \in H$
    with $\|R_t(\varphi)\|_{\cY^\prime} \leq \hnorm{\varphi}$,
    which implies the assertion $R_t\in\cL(H;\cY^\prime)$ with $\|R_t\|_{\cL(H;\cY^\prime)} \leq 1$ for all $t\in\bbT$.

    By Lemma~\ref{lem:tensor}~(i) there exists a unique continuous extension
    $R_{s,t} \in \cL(H \xpi H; \cY^\prime \xpi \cY^\prime)$ of $R_{s,t} \colon \cX\otimes\cX \to \cY^\prime \otimes \cY^\prime$
    defined in \eqref{eq:Rst} for $s$, $t\in\bbT$. The fact that $\cY^{(\varepsilon)\prime}$ is isometrically isomorphic to $\cY^\prime \xpi \cY^\prime$,
    cf.~Lemma~\ref{lem:dualequal}, completes the proof.
\end{proof}

\subsection{A weak It\^{o} isometry}\label{subsec:weakito}

In this subsection the diagonal trace operator is used to formulate
an isometry for the expectation of
the product of two weak stochastic integrals driven by the same L\'evy process.
This isometry is an essential component in the derivation of the deterministic variational
problems for the second moment and the covariance in Sections~\ref{section:moment} and~\ref{section:covariance}.

\begin{lemma}\label{lem:weakito}
    For a predictable process
    $\Phi \in L^2(\Omega\times\bbT; \cL(U;H))$ 
    and the covariance kernel $q \in U^{(2)}$ in~\eqref{eq:qdef}
    the function $\bbE[\Phi(\cdot)\otimes\Phi(\cdot)] q$ is a
    well-defined element in the space $\cW^{(\pi)}$.
    The weak stochastic integral, cf.~\eqref{eq:weakstochint},
    satisfies, for $v_1, v_2 \in \cY$,
    \begin{align*}
        \bbE\biggl[ \int_{0}^{T} \hip{v_1(s)}{\Phi(s) \, \rd L(s)}
                &\int_{0}^{T} \hip{v_2(t)}{\Phi(t) \, \rd L(t)} \biggr]  \\
                &= T_\delta(\bbE[\Phi(\cdot)\otimes\Phi(\cdot)] q)(v_1 \otimes v_2).
    \end{align*}
\end{lemma}

\begin{proof}
    In order to prove that $\bbE[\Phi(\cdot)\otimes\Phi(\cdot)] q$ is a
    well-defined element in the space $\cW^{(\pi)}$, it suffices to show
    that $\Phi(\cdot)\otimes\Phi(\cdot) \in L^1(\Omega; \cL(U^{(\pi)}; \cW^{(\pi)}))$,
    and, hence, $\bbE[\Phi(\cdot)\otimes\Phi(\cdot)] \in \cL(U^{(\pi)}; \cW^{(\pi)})$,
    since $q\in U^{(\pi)}$ by Lemma~\ref{lem:qupiconv}.
    To this end, we estimate
    \begin{align*}
        \| \Phi(\cdot)&\otimes\Phi(\cdot) \|_{L^1(\Omega;\cL(U^{(\pi)};\cW^{(\pi)}))}
             = \bbE \left[\|\Phi(\cdot)\otimes\Phi(\cdot)\|_{\cL(U^{(\pi)};\cW^{(\pi)})}\right]
             = \bbE \left[\|\Phi(\cdot)\|_{\cL(U;\cW)}^2 \right] \\
            &= \bbE\left[ \sup\limits_{\substack{\psi\in U \\ \| \psi\|_U = 1}}
                    \int_0^T \|\Phi(t) \psi\|_H^2 \, \rd t \right]
             \leq \bbE\left[ \int_0^T
                    \sup\limits_{\substack{\psi\in U \\ \| \psi\|_U = 1}} \|\Phi(t) \psi\|_H^2 \, \rd t \right] \\
            &= \bbE \left[ \int_0^T \| \Phi(t) \|_{\cL(U;H)}^2 \, \rd t \right]
             = \| \Phi \|_{L^2(\Omega\times\bbT;\cL(U;H))}^2 < +\infty.
    \end{align*}
    In order to justify that the weak stochastic integrals
    are well-defined, we note that the following embedding holds,
    \[
        L^2(\Omega\times\bbT; \cL(U;H)) \hookrightarrow L^2(\Omega\times\bbT; \lhs(\cH;H))
    \]
    with embedding constant $\sqrt{\tr(Q)} < +\infty$, since
    \begin{align*}
        \|\Phi\|_{L^2(\Omega\times\bbT; \lhs(\cH;H))}^2
            &= \bbE \int_0^T \| \Phi(t) \|_{\lhs(\cH; H)}^2 \, \rd t
             = \bbE \int_0^T \sum\limits_{j\in\cI} \| \Phi(t) f_j \|_{H}^2 \, \rd t \\
            &\leq \bbE \int_0^T \sum\limits_{j\in\cI} \| \Phi(t) \|_{\cL(U;H)}^2 \|f_j \|_{U}^2 \, \rd t \\
            &= \tr(Q) \, \bbE \int_0^T \| \Phi(t) \|_{\cL(U;H)}^2 \, \rd t
             = \tr(Q) \, \|\Phi\|_{L^2(\Omega\times\bbT; \cL(U;H))}^2,
    \end{align*}
    where $f_n := \sqrt{\gamma_n} \, e_n$ and $\cI := \{j\in\bbN : \gamma_j \neq 0\}$ for an eigenbasis
    $(e_n)_{n\in\bbN} \subset U$ of $Q$ with corresponding eigenvalues $(\gamma_n)_{n\in\bbN}$.
    For this reason, the weak stochastic integrals $\int\nolimits_0^T \hip{v_\ell(t)}{\Phi(t) \, \rd L(t)}$
    are well-defined for $v_\ell \in \cY \subset C^0(\bbT; H)$, $\ell \in \{1,2\}$.
    Recalling the definition of the weak stochastic integral in~\eqref{eq:weakstochint} yields the equality
    \begin{align*}
        \int_0^T \hip{v_\ell(t)}{\Phi(t) \, \rd L(t)} = \int_0^T \Psi_\ell(t) \, \rd L(t), \quad \ell=1,2,
    \end{align*}
    where for $\ell\in\{1,2\}$ the stochastic process $\Psi_\ell \in L^2(\Omega\times\bbT; \cL(U;\bbR))$ is defined by
    \begin{align*}
        \Psi_\ell(t) \colon z \mapsto \hip{v_\ell(t)}{\Phi(t) z} \quad \forall z\in \cH
    \end{align*}
    for all $t \in \bbT$.
    Applying It\^o's isometry, see \cite[Corollary~8.17]{peszat2007}, along with the polarisation
    identity, yields
    \begin{align*}
        \bbE\biggl[ \int_{0}^{T} \Psi_1(t) \, \rd L(t)
                \int_{0}^{T} \Psi_2(t) \, \rd L(t) \biggr]
                = \int_0^T \bbE\left[ \langle \Psi_1(t), \Psi_2(t) \rangle_{\lhs(\cH;\bbR)} \right] \rd t,
    \end{align*}
    where $\langle\cdot,\cdot\rangle_{\lhs(\cH;\bbR)}$ denotes the Hilbert--Schmidt inner product, i.e.,
    \begin{align*}
        \langle \tilde{\Phi}, \tilde{\Psi} \rangle_{\lhs(\cH;\bbR)}
             = \sum\limits_{n\in\bbN} \tilde{\Phi}(\tilde{f}_n) \, \tilde{\Psi}(\tilde{f}_n)
    \end{align*}
    for $\tilde\Phi$, $\tilde\Psi\in\lhs(\cH;\bbR)$,
    where $(\tilde{f}_n)_{n\in\bbN}$ is an orthonormal basis of $\cH$.
    By choosing the orthonormal basis $(f_j)_{j\in\cI}$ from above we obtain
    {\allowdisplaybreaks
    \begin{align*}
        \bbE\biggl[ \int_{0}^{T} \hip{v_1(s)}{&\, \Phi(s) \, \rd L(s)}
                    \int_{0}^{T} \hip{v_2(t)}{\Phi(t) \, \rd L(t)} \biggr]  \\
            &= \int_0^T \bbE\left[ \langle \Psi_1(t), \Psi_2(t) \rangle_{\lhs(\cH;\bbR)} \right] \, \rd t \\
            &= \int_0^T \bbE\biggl[ \sum\limits_{j\in\cI} \hip{v_1(t)}{\Phi(t) f_j} \hip{v_2(t)}{\Phi(t) f_j} \biggr] \, \rd t\\
            &= \int_0^T \bbE\biggl[ \sum\limits_{n\in\bbN} \gamma_n \hip{v_1(t)}{\Phi(t) e_n} \hip{v_2(t)}{\Phi(t) e_n} \biggr] \, \rd t\\
            &= \int_0^T \bbE\biggl[ \sum\limits_{n\in\bbN} \htip{v_1(t)\otimes v_2(t)}{[\Phi(t) \otimes \Phi(t)] \gamma_n (e_n \otimes e_n)} \biggr] \, \rd t \\
            &= \int_0^T \bbE\biggl[ \langle v_1(t)\otimes v_2(t),
                            [\Phi(t) \otimes \Phi(t)] \sum\limits_{n\in\bbN} \gamma_n (e_n \otimes e_n) \rangle_{H^{(2)}} \biggr] \, \rd t \\
            &= \int_0^T \bbE\left[ \htip{v_1(t)\otimes v_2(t)}{[\Phi(t) \otimes \Phi(t)] q} \right] \, \rd t.
    \end{align*} }
    By Proposition~\ref{prop:deltaggq} the diagonal trace
    $T_\delta(\bbE[\Phi(\cdot) \otimes \Phi(\cdot)] q)(v_1\otimes v_2)$
    is well-defined, since $\bbE[\Phi(\cdot) \otimes \Phi(\cdot)] q \in \cW^{(\pi)}$.
    With the introduced notion of the operator $T_\delta$ we can rewrite the above expression as
    \begin{align*}
      \bbE&\biggl[ \int_{0}^{T} \hip{v_1(s)}{\Phi(s) \, \rd L(s)}
                   \int_{0}^{T} \hip{v_2(t)}{\Phi(t) \, \rd L(t)} \biggr]  \\
            &= \int_0^T \htip{v_1(t)\otimes v_2(t)}{\bbE[\Phi(t) \otimes \Phi(t)] q} \, \rd t
             = T_\delta(\bbE[\Phi(\cdot) \otimes \Phi(\cdot)] q)(v_1\otimes v_2),
    \end{align*}
    which completes the proof.
\end{proof}

\section{The second moment}\label{section:moment}

After having introduced the stochastic partial differential equation of interest and its mild solution
in Section~\ref{section:spde}, the aim of this section is to derive a well-posed deterministic
variational problem, which is satisfied by the second moment of the mild solution.

The second moment of a random variable $Y\in L^2(\Omega; H_1)$
taking values in a Hilbert space $H_1$ is denoted
by $\bbM^{(2)}Y := \bbE[ Y\otimes Y]$. 
We recall the Bochner spaces $\cW = L^2(\bbT; H)$, $\cX = L^2(\bbT; V)$
and $\cY = L^2(\bbT; V) \cap H_{0,\{T\}}^1(\bbT;V^*)$.
It follows immediately from the definition of the mild solution that
its second moment is an element of the tensor space $\cW^{(2)}$.
Under the assumptions made above we can prove even more regularity.

\begin{theorem}\label{thm:regularsecond} 
    Let Assumption~\ref{ass:spde}~\ref{ass:spde-i}--\ref{ass:spde-iv} be satisfied. Then the second moment of the
    mild solution $X$ defined in~\eqref{eq:mildsol} satisfies $\bbM^{(2)}X \in \cX^{(\pi)} = \cX \xpi \cX$.
\end{theorem}

\begin{proof}
    First, we remark that
    \[
        \| \bbM^{(2)} X \|_{\cX^{(\pi)}} = \| \bbE[X \otimes X] \|_{\cX^{(\pi)}} \leq \bbE \| X \otimes X \|_{\cX^{(\pi)}} = \bbE\left[\| X \|_{\cX}^2\right].
    \]
    Hence, we may estimate as follows:
    \begin{align*}
        \| & \bbM^{(2)} X \|_{\cX^{(\pi)}}
             \leq \bbE \int_0^T \Bigl \| S(t) X_0 + \int_0^t S(t-s) G(X(s)) \, \rd L(s) \Bigr \|_{V}^2 \, \rd t \\
            &\leq 2 \, \bbE \int_0^T \left[ \| S(t) X_0 \|_{V}^2 + \Bigl \| \int_0^t S(t-s) G(X(s)) \, \rd L(s) \Bigr\|_{V}^2 \right] \, \rd t \\
            &= 2 \, \bbE\biggl[ \int_0^T \| A^{\frac{1}{2}} S(t) X_0\|_H^2 \, \rd t \biggr]
              + 2  \int_0^T \bbE \, \Bigl\| \int_0^t A^{\frac{1}{2}} S(t-s) G(X(s)) \, \rd L(s) \Bigr\|_H^2 \, \rd t.
    \end{align*}
    Since the generator $-A$ of the semigroup $(S(t), t\geq 0)$
    is self-adjoint and negative definite, we can bound the first
    integral from above by using the inequality
    \begin{equation}\label{eq:half}
        \int_0^T \| A^{\frac{1}{2}} S(t) \varphi \|_H^2 \, \rd t \leq \frac{1}{2} \| \varphi\|_H^2, \qquad \varphi\in H,
    \end{equation}
    and for the second term we use It\^o's isometry, cf.~\cite[Corollary~8.17]{peszat2007},
    as well as the affine structure of the operator $G$ to obtain
    \begin{align*}
        \| \bbM^{(2)} X \|_{\cX^{(\pi)}}
                &\leq \bbE \| X_0\|_H^2 + 2 \int_0^T \bbE \int_0^t \| A^{\frac{1}{2}} S(t-s) G(X(s)) \|_{\lhs(\cH; H)}^2 \, \rd s \, \rd t \\
                &\leq \bbE \| X_0\|_H^2
                    + 4 \int_0^T \int_0^t \| A^{\frac{1}{2}} S(t-s) G_2 \|_{\lhs(\cH; H)}^2 \, \rd s\,\rd t \\
                &\quad + 4
                    \int_0^T \bbE \int_0^t \| A^{\frac{1}{2}} S(t-s) G_1(X(s)) \|_{\lhs(\cH; H)}^2 \, \rd s\,\rd t.
    \end{align*}
    By Assumption~\ref{ass:spde}~\ref{ass:spde-i}--\ref{ass:spde-iii} as well as Theorem~\ref{thm:exunimild}
    there exists a regularity exponent $r\in[0,1]$ such that
    the mild solution satisfies $X\in L^\infty(\bbT; L^2(\Omega; \Hr))$.
    In addition, by Assumption~\ref{ass:spde}~\ref{ass:spde-iv} it holds
    that $A^{1/2} S(\cdot) G_1 \in L^2(\bbT; \cL(\Hr; \lhs(\cH; H)))$.
    Then we estimate as follows,
    \begin{align*}
        \| \bbM^{(2)} X \|_{\cX^{(\pi)}}
                &\leq \bbE \| X_0\|_H^2
                    + 4 \sum\limits_{n\in\bbN} \int_0^T \int_0^t
                                \| A^{\frac{1}{2}} S(t-s) G_2 f_n\|_{H}^2 \, \rd s\,\rd t \\
                &\quad + 4 \int_0^T \int_0^t \| A^{\frac{1}{2}} S(t-s) G_1 \|_{\cL(\Hr; \lhs(\cH; H))}^2
                           \bbE \| X(s)\|_{\Hr}^2 \, \rd s\,\rd t
    \end{align*}
    for an orthonormal basis $(f_n)_{n\in\bbN}$ of $\cH$.
    Applying~\eqref{eq:half} again
    with upper integral bound $t$ instead of $T$ yields
    \begin{align*}
        \| \bbM^{(2)} X \|_{\cX^{(\pi)}}
                &\leq \| X_0\|_{L^2(\Omega; H)}^2
                    + 2 T \| G_2 \|_{\lhs(\cH; H)}^2 \\
                &\quad + 4 T \| X \|_{L^\infty(\bbT; L^2(\Omega; \Hr))}^2
                            \|A^{\frac{1}{2}} S(\cdot) G_1\|_{L^2(\bbT; \cL(\Hr; \lhs(\cH; H)))}^2,
    \end{align*}
    which is finite under our assumptions and completes the proof.
\end{proof}

We define the bilinear form $\cB\colon \cX \times \cY \to \bbR$ by
\begin{equation}\label{eq:cB}
    \cB(u,v) := \int_0^T \vdpt{u(t)}{(-\partial_t + A^*) v(t)} \, \rd t, 
    \qquad u \in \cX, v \in \cY,
\end{equation}
and the mean function $m$ of the mild solution $X$
in~\eqref{eq:mildsol} by
\begin{equation}\label{eq:meanfct}
    m(t) := \bbE X(t) = S(t) \bbE X_0, \quad t \in \bbT.
\end{equation}
Note that due to the mean zero property of the stochastic integral
the mean function depends only on the initial value
$X_0$ and not on the operator $G$. Furthermore, applying inequality~\eqref{eq:half}
shows the regularity $m\in\cX$, and $m$ can be interpreted
as the unique function satisfying
\begin{equation}\label{eq:mvar}
    m\in\cX : \quad \cB(m,v) = \hip{\bbE X_0}{v(0)} \quad \forall v \in \cY.
\end{equation}
Well-posedness of this problem 
follows from~\cite[Theorem~2.3]{schwab2013}.

In addition, we introduce the operator $\bbB\colon\cX \to \cY^\prime$
associated with the bilinear form $\cB$, i.e., $\bbB u := \cB(u,\cdot) \in \cY^\prime$
for $u\in\cX$. Then, this linear operator is bounded, $\bbB \in \cL(\cX,\cY^\prime)$
and $\bbB \otimes \bbB \colon \cX \otimes \cX \to \cY^\prime \otimes \cY^\prime$
defined by 
\[
    (\bbB \otimes \bbB)\Bigl(\sum\limits_{k=1}^{N} u_k^1 \otimes u_k^2 \Bigr) := \sum\limits_{k=1}^{N} \bbB u_k^1 \otimes \bbB u_k^2
            = \sum\limits_{k=1}^{N} \cB(u_k^1,\cdot) \otimes \cB(u_k^2,\cdot)
\]
admits a unique extension to a bounded linear operator $\bbB^{(\pi)} \in \cL(\cX^{(\pi)}; (\cY^\prime)^{(\pi)})$
with $\bbB^{(\pi)} = \bbB \otimes \bbB$ on $\cX \otimes \cX$ and
$\|\bbB^{(\pi)}\|_{\cL(\cX^{(\pi)};(\cY^\prime)^{(\pi)})} = \|\bbB \|_{\cL(\cX;\cY^\prime)}^2$
by~Lemma~\ref{lem:tensor}~\ref{lem:tensor-i}.
With these definitions and preliminaries we are now able to show
that the second moment of the mild solution solves a
deterministic variational problem.

\begin{theorem}\label{thm:deterministic}
    Let all conditions of Assumption~\ref{ass:spde} be satisfied and let
    $X$ be the mild solution to~\eqref{eq:spdeall}.
    Then the second moment $\bbM^{(2)}X \in \cX^{(\pi)}$ solves the following variational problem
    \begin{equation} \label{eq:deterministic}
        u\in\cX^{(\pi)} : \quad \cBpi{u}{v} = f(v) \quad \forall v \in \cY^{(\varepsilon)},
    \end{equation}
    where for $u\in\cX^{(\pi)}$ and $v\in\cY^{(\varepsilon)}$
    \begin{align}
        \cBpi{u}{v} &:= \bbB^{(\pi)}(u)v - T_\delta((G_1 \otimes G_1) (u) q) v, \label{eq:cBpi} \\
        f(v)        &:= R_{0,0}\bigl(\bbM^{(2)} X_0\bigr)v
                            + T_\delta( (G_1(m) \otimes G_2) q)v \notag \\
                    &\quad  + T_\delta( (G_2 \otimes G_1(m)) q)v + T_\delta((G_2 \otimes G_2) q)v \notag
    \end{align}
    with the operators $T_\delta$ and $R_{0,0}$ defined in~\eqref{eq:Tdelta} and~\eqref{eq:Rst}
    and the mean function $m\in\cX$ in~\eqref{eq:meanfct}.
\end{theorem}

\begin{proof}
    First, we remark that $\cBpi{u}{v}$ is well-defined for $u\in\cX^{(\pi)}$ and $v\in\cY^{(\varepsilon)}$,
    since the tensor spaces $\cY^\prime \xpi \cY^\prime$ and $(\cY \xeps \cY)^\prime$ are isometrically isomorphic
    by Lemma~\ref{lem:dualequal} and, hence, $\bbB^{(\pi)}u - T_\delta((G_1 \otimes G_1) (u) q) \in \cY^{(\varepsilon)\prime}$
    for all $u\in\cX^{(\pi)}$ by the definition of $\bbB^{(\pi)}$ and Proposition~\ref{prop:deltaggq}. 

    Let $v_1, v_2 \in C_{0,\{T\}}^1(\bbT; \cD(A^*)) = \{\phi\in C^1(\bbT; \cD(A^*)) : \phi(T) = 0\}$.
    Then, we obtain
    \begin{align*}
        \bbB^{(\pi)} \bigl(&\bbM^{(2)} X\bigr)(v_1 \otimes v_2 )
             = \bbB^{(\pi)} (\bbE[ X\otimes X])(v_1 \otimes v_2 )
             = \bbE \bigl[ \bbB^{(\pi)} ( X\otimes X)(v_1 \otimes v_2 ) \bigr] \\
            &= \bbE \bigl[ (\bbB(X) \otimes \bbB(X))(v_1 \otimes v_2) \bigr]
             = \bbE [ \cB(X,v_1) \, \cB(X,v_2) ] \\
            &= \bbE \left[ \lthip{X}{(-\partial_t+A^*)v_1} \lthip{X}{(-\partial_t+A^*)v_2} \right].
    \end{align*}
    Due to the regularity of $v_1$ and $v_2$ we may take
    the inner product on $L^2(\bbT; H)$ in this calculation.
    Now, since $X$ is the mild solution of~\eqref{eq:spdeall},
    Lemma~\ref{lem:weakequal} yields
    \begin{align*}
        \bbB^{(\pi)} \bigl(&\bbM^{(2)} X \bigr)(v_1 \otimes v_2)
             = \bbE\biggl[ \Bigl(\hip{X_0}{v_1(0)} + \int_0^T \hip{v_1(s)}{G(X(s)) \, \rd L(s)} \Bigr) \\
            &\qquad  \cdot \Bigl(\hip{X_0}{v_2(0)} + \int_0^T \hip{v_2(t)}{G(X(t)) \, \rd L(t)} \Bigr) \biggr] \\
            &= \bbE\left[ \hip{X_0}{v_1(0)} \hip{X_0}{v_2(0)} \right] \\
            &\quad + \bbE\Bigl[ \hip{X_0}{v_1(0)} \int_0^T \hip{v_2(t)}{G(X(t)) \, \rd L(t)} \Bigr] \\
            &\quad + \bbE\Bigl[ \hip{X_0}{v_2(0)} \int_0^T \hip{v_1(s)}{G(X(s)) \, \rd L(s)} \Bigr] \\
            &\quad + \bbE\Bigl[ \int_0^T \hip{v_1(s)}{G(X(s)) \, \rd L(s)}
                                \int_0^T \hip{v_2(t)}{G(X(t)) \, \rd L(t)} \Bigr].
    \end{align*}
    The $\cF_0$-measurability of $X_0\in L^2(\Omega; H)$, along with
    the independence of the stochastic integral with respect to $\cF_0$
    and its mean zero property imply that the second and the
    third term vanish: For $\ell\in\{1,2\}$ we define the $\lhs(\cH; \bbR)$-valued stochastic process
    $\Psi_{\ell}$ by
    \[
        \Psi_{\ell}(t) \colon w \mapsto \hip{v_{\ell}(t)}{G(X(t)) w} \quad \forall w\in \cH
    \]
    for $t\in\bbT$, $\bbP$-almost surely. Then we obtain
    $\|\Psi_{\ell}(t)\|_{\lhs(\cH; \bbR)}^2 = \chnorm{G(X(t))^* v_\ell(t)}^2$
    $\bbP$-almost surely with the adjoint $G(X(t))^*\in\cL(H; \cH)$ of $G(X(t))$ 
    and
    \begin{align*}
        \bbE\Bigl[ \hip{X_0}{v_\ell(0)} & \int_0^T \hip{v_\ell(t)}{G(X(t)) \, \rd L(t)} \Bigr]
            = \bbE\Bigl[ \hip{X_0}{v_\ell(0)} \int_0^T  \Psi_{\ell}(t) \, \rd L(t) \Bigr]  \\
            &= \bbE\biggl[ \hip{X_0}{v_\ell(0)} \bbE\Bigl[ \int_0^T \Psi_{\ell}(t) \, \rd L(t) \, \Big| \, \cF_0 \Bigr] \biggr] = 0
    \end{align*}
    by the definition of the weak stochastic integral, cf.~\cite[p.~151]{peszat2007},
    the independence of the stochastic integral with respect to $\cF_0$
    and the fact that the stochastic integral has mean zero. 
    For the first term we calculate 
    by using the operator $R_{0,0}$ defined in \eqref{eq:Rst} and its continuity 
    $R_{0,0} \in \cL(H^{(\pi)}; \cY^{(\varepsilon)\prime})$, cf.~Lemma~\ref{lem:Rst},
    \begin{align*}
        \bbE[ \hip{X_0}{v_1(0)} & \hip{X_0}{v_2(0)} ]
             = \bbE\left[ R_{0,0}(X_0 \otimes X_0)(v_1 \otimes v_2)\right] \\
            &= R_{0,0}(\bbE[X_0 \otimes X_0])(v_1\otimes v_2) = R_{0,0}\bigl(\bbM^{(2)} X_0 \bigr)(v_1\otimes v_2).
    \end{align*}
    Finally, the predictability of $X$ together with the continuity assumptions on $G$ imply
    the predictability of $G(X)$ and we may use Lemma~\ref{lem:weakito} for the
    last term yielding
    \begin{align*}
        \bbE&\Bigl[ \int_0^T \hip{v_1(s)}{G(X(s)) \, \rd L(s)}
                                \int_0^T \hip{v_2(t)}{G(X(t)) \, \rd L(t)} \Bigr] \\
            &= T_\delta(\bbE[G(X)  \otimes   G(X)]q)(v_1\otimes v_2) \\
            &= T_\delta(\bbE[G_1(X) \otimes G_1(X)]q)(v_1\otimes v_2)
                + T_\delta((\bbE[G_1(X)] \otimes G_2)q)(v_1\otimes v_2) \\
            &\quad + T_\delta((G_2   \otimes \bbE[G_1(X)])q)(v_1\otimes v_2)
                + T_\delta((G_2   \otimes G_2)q)(v_1\otimes v_2)   \\
            &= T_\delta((G_1\otimes G_1)(\bbM^{(2)} X)q)(v_1\otimes v_2)
                + T_\delta((G_1(m)\otimes G_2)q)(v_1\otimes v_2) \\
            &\quad + T_\delta((G_2   \otimes G_1(m))q)(v_1\otimes v_2)
                + T_\delta((G_2   \otimes G_2)q)(v_1\otimes v_2).
    \end{align*}
    Since $C_{0,\{T\}}^1(\bbT; \cD(A^*)) \subset \cY$ is a dense subset, the claim follows.
\end{proof}

\section{Existence and uniqueness}\label{section:exuni}

Before we extend the results of Section~\ref{section:moment}
for the second moment to the covariance of the mild solution
in Section~\ref{section:covariance},
we investigate in this section well-posedness of the variational
problem~\eqref{eq:deterministic} satisfied by the second moment.

To this end, we first take a closer look at the variational problem \eqref{eq:mvar}
satisfied by the mean function $m=\bbE X$ of the solution process $X$.
The bilinear form $\cB$ arising in this problem is known to satisfy an
inf-sup and a surjectivity condition on $\cX \times \cY$,
cf.~the second part of~\cite[Theorem~2.2]{schwab2013}. 

\begin{theorem}\label{thm:beta} 
    For the bilinear form $\cB$ in~\eqref{eq:cB} the following hold:
    \begin{align}
        \beta := \inf\limits_{u\in\cX\setminus\{0\}} \sup\limits_{v \in \cY\setminus\{0\}}
            \frac{\cB(u,v)}{\xnorm{u} \ynorm{v}} &> 0, \label{eq:beta} \\
        \forall v \in \cY\setminus\{0\}: \quad \sup\limits_{u \in \cX} \cB(u,v) &> 0. \notag
    \end{align}
\end{theorem}


For proving well-posedness of the variational problem~\eqref{eq:deterministic}
satisfied by the second moment of the mild solution, we need a lower bound on
the inf-sup constant $\beta$ in~\eqref{eq:beta}.
In order to derive this bound, we first recall the
Ne\v{c}as theorem, cf.~\cite[Theorem~2.2,~p.~422]{devore2009}.

\begin{theorem}[Ne\v{c}as theorem]\label{thm:necas} 
    Let $H_1$ and $H_2$ be two separable Hilbert spaces and $\scrB\colon H_1 \times H_2 \to \bbR$
    a continuous bilinear form. Then the variational problem
    \begin{equation}\label{eq:necasthm}
        u \in H_1 : \quad \scrB(u,v) = f(v) \quad \forall v \in H_2,
    \end{equation}
    admits a unique solution $u\in H_1$ for all $f \in H_2^\prime$, which depends continuously on $f$,
    if and only if the bilinear form $\scrB$ satisfies one of the following equivalent inf-sup conditions:
    \begin{enumerate}
    \item\label{thm:necas-i}
        It holds
        \begin{equation*}
            \inf\limits_{v_1\in H_1\setminus\{0\}} \sup\limits_{v_2 \in H_2\setminus\{0\}}
                \frac{\scrB(v_1, v_2)}{\|v_1 \|_{H_1} \|v_2\|_{H_2}} > 0,
            \quad
            \inf\limits_{v_2\in H_2\setminus\{0\}} \sup\limits_{v_1 \in H_1\setminus\{0\}}
                \frac{\scrB(v_1, v_2)}{\|v_1 \|_{H_1} \|v_2\|_{H_2}} > 0.
        \end{equation*}
    \item\label{thm:necas-ii}
        There exists $\gamma > 0$ such that
        \begin{equation*}
            \inf\limits_{v_1\in H_1\setminus\{0\}} \sup\limits_{v_2 \in H_2\setminus\{0\}}
                \frac{\scrB(v_1, v_2)}{\|v_1 \|_{H_1} \|v_2\|_{H_2}}
            =
            \inf\limits_{v_2 \in H_2\setminus\{0\}} \sup\limits_{v_1\in H_1\setminus\{0\}}
                \frac{\scrB(v_1, v_2)}{\|v_1 \|_{H_1} \|v_2\|_{H_2}}
            = \gamma.
        \end{equation*} 
    \end{enumerate}
    In addition, the solution $u$ of~\eqref{eq:necasthm} satisfies the stability estimate
    \begin{equation*}
        \|u\|_{H_1} \leq \gamma^{-1} \|f\|_{H_2^\prime}.
    \end{equation*}
\end{theorem}

By using the equivalence of the conditions~\ref{thm:necas-i} and~\ref{thm:necas-ii} in the Ne\v{c}as theorem
we are able to calculate a lower bound on $\beta$ in the following lemma.

\begin{lemma}\label{lem:beta} 
    The inf-sup constant $\beta$ in~\eqref{eq:beta} satisfies $\beta\geq 1$.
\end{lemma}

\begin{proof}
    Combining the results of Theorem~\ref{thm:beta} with the equivalence of~\ref{thm:necas-i} and~\ref{thm:necas-ii} in Theorem~\ref{thm:necas}
    yields the equality
    \[
        \beta = \inf\limits_{u\in\cX\setminus\{0\}} \sup\limits_{v\in\cY\setminus\{0\}} \frac{\cB(u,v)}{\xnorm{u} \ynorm{v}}
              = \inf\limits_{v\in\cY\setminus\{0\}} \sup\limits_{u\in\cX\setminus\{0\}} \frac{\cB(u,v)}{\xnorm{u} \ynorm{v}}.
    \]
    To derive a lower bound for $\beta$, we proceed as
    in~\cite{schwab2009, urban2014}.
    Fix $v \in \cY\setminus\{0\}$, and define $u := v - (A^*)^{-1} \partial_t v$, where
    $(A^*)^{-1}$ is the right-inverse of the surjection $A^* \in \cL(V;V^*)$.
    Then $u\in\cX = L^2(\bbT; V)$ since $(A^*)^{-1} \in \cL(V^*; V)$ and we calculate as follows:
    \begin{align*}
        \xnorm{u}^2 &= \int_0^T \vnorm{u(t)}^2 \, \rd t
             = \int_0^T \vdpt{u(t)}{A^* u(t)} \, \rd t\\
            &= \int_0^T \vdpt{v(t) - (A^*)^{-1} \partial_t v(t)}{A^* v(t) - \partial_t v(t)} \, \rd t \\
            &= \int_0^T \vdpt{v(t)}{A^* v(t)} \, \rd t
                + \int_0^T \vdpt{(A^*)^{-1} \partial_t v(t)}{\partial_t v(t)} \, \rd t \\
            &\quad - \int_0^T \vdpt{v(t)}{\partial_t v(t)} \, \rd t
                - \int_0^T \vdpt{(A^*)^{-1} \partial_t v(t)}{A^* v(t)} \, \rd t.
    \end{align*}
    Now the symmetry of the inner product $\vip{\cdot}{\cdot}$ on $V$ yields
    \begin{align*}
        \vdpt{(A^*)^{-1} \partial_t v(t)}{A^* v(t)}
            &= \vip{(A^*)^{-1} \partial_t v(t)}{v(t)}
             = \vip{v(t)}{(A^*)^{-1} \partial_t v(t)} \\
            &= \vdpt{v(t)}{\partial_t v(t)},
    \end{align*}
    and by inserting the identity $A^* (A^*)^{-1}$,
    using $\tfrac{\rd}{\rd t} \hnorm{v(t)}^2 = 2 \vdpt{v(t)}{\partial_t v(t)}$
    and $v(T) = 0$ we obtain
    \begin{align*}
        \xnorm{u}^2
            &= \xnorm{v}^2 + \xnorm{(A^*)^{-1} \partial_t v}^2
                -  \int_0^T 2 \vdpt{v(t)}{\partial_t v(t)} \, \rd t \\
            &= \xnorm{v}^2 + \xnorm{(A^*)^{-1} \partial_t v}^2
                + \hnorm{v(0)}^2 \\
            &\geq \xnorm{v}^2 + \xnorm{(A^*)^{-1} \partial_t v}^2
             = \xnorm{v}^2 + \| \partial_t v\|_{L^2(\bbT;V^*)}^2 = \ynorm{v}^2.
    \end{align*}
    In the last line we used that $\vdnorm{w} = \vnorm{(A^*)^{-1} w}$ for every $w\in V^*$,
    since
    \begin{align*}
        \vdnorm{w} &= \sup\limits_{v\in V\setminus\{0\}} \frac{\vdpt{v}{w}}{\vnorm{v}} \\
                   &= \sup\limits_{v\in V\setminus\{0\}} \frac{\vdpt{v}{A^*((A^*)^{-1} w)}}{\vnorm{v}}
                    = \sup\limits_{v\in V\setminus\{0\}} \frac{\vip{v}{(A^*)^{-1} w}}{\vnorm{v}}
                    = \vnorm{(A^*)^{-1} w}.
    \end{align*}
    Hence, we obtain for any fixed $v\in \cY$ and $u = v - (A^*)^{-1} \partial_t v$ that
    $\xnorm{u} \geq \ynorm{v}$. In addition, we estimate
    \begin{align*}
        \cB(u,v) &= \int_0^T \vdpt{u(t)}{(-\partial_t + A^*) v(t)} \, \rd t \\
                 &= \int_0^T \vdpt{v(t) - (A^*)^{-1} \partial_t v(t)}{A^*(v(t) - (A^*)^{-1} \partial_t v(t))} \, \rd t\\
                 &= \int_0^T \vnorm{v(t) - (A^*)^{-1} \partial_t v(t)}^2 \, \rd t
                  = \xnorm{v - (A^*)^{-1} \partial_t v}^2
                  = \xnorm{u}^2 \geq \xnorm{u} \ynorm{v}
    \end{align*}
    and, therefore,
    \[
        \sup\limits_{w\in\cX\setminus\{0\}} \frac{\cB(w,v)}{\xnorm{w}} \geq \ynorm{v} \quad \forall v \in \cY.
    \]
    This shows the assertion
    \[
        \beta = \inf\limits_{w\in\cX\setminus\{0\}} \sup\limits_{v\in\cY\setminus\{0\}} \frac{\cB(w,v)}{\xnorm{w} \ynorm{v}}
              = \inf\limits_{v\in\cY\setminus\{0\}} \sup\limits_{w\in\cX\setminus\{0\}} \frac{\cB(w,v)}{\xnorm{w} \ynorm{v}}
              \geq 1. \qedhere
    \]
\end{proof}

The result on the inf-sup constant $\beta$ in Lemma~\ref{lem:beta} above
can be formulated in terms of the operator $\bbB\in\cL(\cX;\cY^\prime)$ associated with the bilinear
form $\cB$ as follows: For every $u\in\cX$ it holds
\begin{equation}\label{eq:Blower}
    \|\bbB u\|_{\cY^\prime} = \sup\limits_{v\in\cY\setminus\{0\}} \frac{\cB(u,v)}{\ynorm{v}}
                            \geq \xnorm{u},
\end{equation}
i.e., $\bbB$ is injective and by Theorem~\ref{thm:beta} also surjective and, hence,
boundedly invertible with $\| \bbB^{-1} \|_{\cL(\cY^\prime; \cX)} \leq 1$.

These preliminary observations on the operator $\bbB$ associated with the
bilinear form $\cB$ yield the following result on the operator
$\bbB^{(\pi)} = \bbB \xpi \bbB$
mapping from the tensor product space $\cX^{(\pi)}$
to the tensor product space $(\cY^\prime)^{(\pi)}$.

\begin{lemma}\label{lem:Bpi}
    The unique operator $\bbB^{(\pi)} \in \cL(\cX^{(\pi)}; (\cY^\prime)^{(\pi)})$ satisfying
    $\bbB^{(\pi)}(u^1 \otimes u^2) = \bbB u^1 \otimes \bbB u^2$ for all $u^1, u^2 \in \cX$
    is injective and, moreover, it holds
    \begin{equation}\label{eq:Bpilower}
        \| \bbB^{(\pi)}(u) \|_{(\cY^\prime)^{(\pi)}} \geq \| u\|_{\cX^{(\pi)}}
    \end{equation}
    for all $u\in\cX^{(\pi)}$.
\end{lemma}

\begin{proof}
    Let $u \in \cX \otimes \cX$ and $\sum\nolimits_{k=1}^{N} u_k^1 \otimes u_k^2$ be a representation of $u$
    and $\sum\nolimits_{\ell=1}^{M} f_\ell^1 \otimes f_\ell^2$ be a representation of $\bbB^{(\pi)} u$.

    Since $\bbB$ is boundedly invertible, $\sum\nolimits_{\ell=1}^{M} \bbB^{-1} f_\ell^1 \otimes \bbB^{-1} f_\ell^2$
    is a well-defined element in $\cX \otimes \cX$ and, furthermore, it is a representation of $u$, since
    \begin{align*}
        u  &= \sum\limits_{k=1}^{N} u_k^1 \otimes u_k^2
            = \sum\limits_{k=1}^{N} \left(\bbB^{-1} \bbB u_k^1\right) \otimes \left(\bbB^{-1} \bbB u_k^2\right)
            = \left( \bbB^{-1} \otimes \bbB^{-1} \right) \biggl( \sum\limits_{k=1}^{N} \bbB u_k^1 \otimes \bbB u_k^2 \biggr) \\
           &= \left( \bbB^{-1} \otimes \bbB^{-1} \right) \bigl( \bbB^{(\pi)} u \bigr)
            = \left( \bbB^{-1} \otimes \bbB^{-1} \right) \biggl( \sum\limits_{\ell=1}^{M} f_\ell^1 \otimes f_\ell^2 \biggr)
            = \sum\limits_{\ell=1}^{M} \bbB^{-1} f_\ell^1 \otimes \bbB^{-1} f_\ell^2.
    \end{align*}
    With this observation we can estimate
    \[
        \|u\|_{\cX^{(\pi)}} \leq \sum\limits_{\ell=1}^{M} \|\bbB^{-1} f_\ell^1\|_{\cX} \|\bbB^{-1} f_\ell^2\|_{\cX}
            \leq \sum\limits_{\ell=1}^{M} \|f_\ell^1\|_{\cY^\prime} \|f_\ell^2\|_{\cY^\prime},
    \]
    since $\|\bbB^{-1}\|_{\cL(\cY^\prime; \cX)} \leq 1$.
    This calculation shows $\|u\|_{\cX^{(\pi)}} \leq \|\bbB^{(\pi)} u \|_{(\cY^\prime)^{(\pi)}}$ for all $u\in\cX^{(\pi)}$ and the assertion is proven.
\end{proof}

By using this lemma together with the properties of the operator $T_\delta$, which we have derived in Section~\ref{subsec:diagonal},
we now prove well-posedness of the variational problem satisfied by the second moment of the mild solution.

\begin{theorem}\label{thm:wellposed} 
    Suppose that
    \begin{equation}\label{eq:Grequire}
        \| G_1 \|_{\cL(V; \lhs(\cH;H))} < 1.
    \end{equation}
    Then the variational problem
    \begin{equation}\label{eq:wvariational}
        w \in \cX^{(\pi)} : \quad \cBpi{w}{v} = f(v) \quad \forall v \in \cY^{(\varepsilon)}
    \end{equation}
    admits at most one solution $w\in\cX^{(\pi)}$ for every $f\in \cY^{(\varepsilon)\prime}$.
    In particular, there exists a unique solution $u\in\cX^{(\pi)}$ satisfying~\eqref{eq:deterministic}.
\end{theorem}

\begin{proof}
    It suffices to show that only $u=0$ solves the homogeneous problem
    \[
        u \in \cX^{(\pi)} : \quad \cBpi{u}{v} = 0 \quad \forall v \in \cY^{(\varepsilon)}.
    \]
    For this purpose, let $u\in\cX^{(\pi)}$ be a solution to the homogeneous problem.
    Then it holds
    \[
        0 = \cBpi{u}{v} = \bbB^{(\pi)}(u)v - T_\delta((G_1\otimes G_1)(u)q) v
    \]
    for all $v\in\cY^{(\varepsilon)}$ and, hence,
    \[
        \| \bbB^{(\pi)} u - T_\delta((G_1\otimes G_1)(u)q) \|_{\cY^{(\varepsilon)\prime}} = 0
    \]
    and we calculate by using
    the estimate~\eqref{eq:Bpilower} of Lemma~\ref{lem:Bpi} as well as Lemma~\ref{lem:dualequal}
    as follows,
    \begin{align*}
        \|u\|_{\cX^{(\pi)}} &\leq \|\bbB^{(\pi)} u\|_{(\cY^\prime)^{(\pi)}} = \|\bbB^{(\pi)} u\|_{\cY^{(\varepsilon)\prime}} \\
            &\leq \| \bbB^{(\pi)} u - T_\delta((G_1\otimes G_1)(u)q) \|_{\cY^{(\varepsilon)\prime}} + \| T_\delta((G_1\otimes G_1)(u)q) \|_{\cY^{(\varepsilon)\prime}} \\
            &= \| T_\delta((G_1\otimes G_1)(u)q) \|_{\cY^{(\varepsilon)\prime}}
    \end{align*}
    In addition, Proposition~\ref{prop:deltaggq} and estimate~\eqref{eq:ggqvv} in Lemma~\ref{lem:ggqvv}~\ref{lem:ggqvv-i} yield
    \begin{align*}
        \|u\|_{\cX^{(\pi)}}
            &\leq \| T_\delta \|_{\cL(\cW^{(\pi)};\cY^{(\varepsilon)\prime})} \|(G_1\otimes G_1)(u)q\|_{\cW^{(\pi)}} \\
            &\leq \|(G_1\otimes G_1)(\cdot)q\|_{\cL(\cX^{(\pi)};\cW^{(\pi)})} \|u\|_{\cX^{(\pi)}}
             \leq \|G_1\|_{\cL(V;\lhs(\cH;H))}^2 \|u\|_{\cX^{(\pi)}}.
    \end{align*} 
    Therefore, $u=0$, if $G_1$ satisfies Condition~\eqref{eq:Grequire}, and the variational problem~\eqref{eq:wvariational} has at most one solution.
    Under Assumption~\ref{ass:spde} on~$X_0$ and the affine operator~$G(\cdot) = G_1(\cdot) + G_2$ there exists
    a unique (up to modification) mild solution~$X$ to the stochastic partial differential equation~\eqref{eq:spdeall} with
    second moment~$\bbM^{(2)} X \in \cX^{(\pi)}$ satisfying the variational problem~\eqref{eq:deterministic},
    cf.~Theorems~\ref{thm:exunimild},~\ref{thm:regularsecond}, and~\ref{thm:deterministic}.
    Therefore, we obtain existence and uniqueness of a solution to~\eqref{eq:wvariational} for the right-hand side
    \begin{align*}
        f(v)        &= R_{0,0}\bigl(\bbM^{(2)} X_0\bigr)v
                            + T_\delta( (G_1(m) \otimes G_2) q)v \\
                    &\quad  + T_\delta( (G_2 \otimes G_1(m)) q)v + T_\delta((G_2 \otimes G_2) q)v,
    \end{align*}
    where $m=\bbE X$ and the variational problem~\eqref{eq:deterministic} is well-posed.
\end{proof}

To conclude, we have shown in this section that there exists a variational problem
that has the second moment of the mild solution~\eqref{eq:mildsol} as its unique solution.
%
%
%

\section{From the second moment to the covariance}\label{section:covariance}

In the previous sections, we have seen that the second moment $\bbM^{(2)} X$
of the mild solution $X$ to the stochastic partial differential
equation~\eqref{eq:spdeall} satisfies a well-posed deterministic
variational problem.
As a consequence of this result we derive another deterministic problem in this section,
which is satisfied by the covariance $\cov(X)$ of the solution process.
For this purpose, we remark first that
\begin{align*}
    \cov(X) &= \bbE\left[ (X - \bbE X) \otimes (X - \bbE X) \right] \\
        &= \bbE\left[ (X\otimes X) - (\bbE X \otimes X) - (X \otimes \bbE X) + (\bbE X \otimes \bbE X) \right] \\
        &= \bbM^{(2)} X - \bbE X \otimes \bbE X
\end{align*}
and $\cov(X) \in \cX^{(\pi)}$,
since $\bbM^{(2)} X \in \cX^{(\pi)}$ by Theorem~\ref{thm:regularsecond} and $m=\bbE X \in \cX$.
By using this relation we can immediately deduce
the following result for the covariance $\cov(X)$
of the mild solution.

\begin{theorem}\label{thm:covariance} 
    Let all conditions of Assumption~\ref{ass:spde} be satisfied and let
    $X$ be the mild solution to~\eqref{eq:spdeall}.
    Then the covariance $\cov(X)\in\cX^{(\pi)}$ solves the well-posed problem
    \begin{equation}\label{eq:covariance}
        u\in\cX^{(\pi)} : \quad
            \cBpi{u}{v} = g(v) \quad \forall v \in \cY^{(\varepsilon)}
    \end{equation}
    with $\widetilde{\cB}^{(\pi)}$ as in~\eqref{eq:cBpi} and
    for $v \in \cY^{(\varepsilon)}$
    \[
        g(v) := R_{0,0}(\cov(X_0))v + T_\delta( (G(m) \otimes G(m)) q)v,
    \]
    where $T_\delta$ and $R_{0,0}$ are the operators defined in~\eqref{eq:Tdelta} and~\eqref{eq:Rst}
    and $m\in\cX$ denotes the mean function introduced in~\eqref{eq:meanfct}.
\end{theorem}

\begin{proof}
    The covariance of the mild solution satisfies that $\cov(X) = \bbM^{(2)} X - \bbE X \otimes \bbE X$
    by the remark above.
    By using the result of Theorem~\ref{thm:deterministic} for the second moment $\bbM^{(2)} X$
    as well as~\eqref{eq:mvar} for the mean function $m = \bbE X$ we calculate for $v_1, v_2\in\cY$:
    \begin{align*}
        \widetilde{\cB}&^{(\pi)}(\cov(X), v_1 \otimes v_2)
             = \cBpi{\bbM^{(2)} X}{v_1 \otimes v_2} - \cBpi{\bbE X \otimes \bbE X}{v_1 \otimes v_2} \\
            &= f(v_1 \otimes v_2)  - \bbB^{(\pi)}(m \otimes m)(v_1 \otimes v_2)
                 + T_\delta( (G_1(m) \otimes G_1(m)) q)(v_1 \otimes v_2) \\
            &= R_{0,0}\bigl(\bbM^{(2)} X_0\bigr)(v_1 \otimes v_2)
                + T_\delta( (G_2 \otimes G_2) q)(v_1 \otimes v_2) \\
            &\quad + T_\delta( (G_1(m) \otimes G_2) q)(v_1 \otimes v_2) as
                + T_\delta( (G_2 \otimes G_1(m)) q)(v_1 \otimes v_2) \\
            &\quad - \hip{\bbE X_0}{v_1(0)} \hip{\bbE X_0}{v_2(0)}
                 + T_\delta( (G_1(m) \otimes G_1(m)) q)(v_1 \otimes v_2) \\
            &= R_{0,0}\bigl(\bbM^{(2)} X_0\bigr)(v_1 \otimes v_2)
                - R_{0,0} (\bbE X_0 \otimes \bbE X_0)(v_1 \otimes v_2) \\
            &\quad + T_\delta( (G(m) \otimes G(m)) q)(v_1 \otimes v_2).
    \end{align*}
    Hence,
    \[
        \cBpi{\cov(X)}{v_1 \otimes v_2} = g(v_1 \otimes v_2) \quad \forall v_1, v_2 \in \cY
    \]
    and this observation completes the proof, since the subset
    $\operatorname{span}\{v_1 \otimes v_2 : v_1, v_2 \in\cY\} \subset \cY^{(\varepsilon)}$
    is dense and 
    well-posedness of~\eqref{eq:covariance} follows from the existence of
    the mild solution $X$ to \eqref{eq:spdeall} as well as its covariance $\cov(X)\in\cX^{(\pi)}$
    and Theorem~\ref{thm:wellposed}.
\end{proof}

\begin{remark}
    Theorem~\ref{thm:covariance} shows that, if only the covariance of the mild solution
    to the stochastic partial differential equation~\eqref{eq:spdeall}
    needs to be computed, then one can do this by solving
    sequentially two deterministic variational problems:
    first, the more or less standard parabolic problem~\eqref{eq:mvar} for the mean function
    and afterwards problem~\eqref{eq:covariance} for the covariance,
    which is posed on non-reflexive tensor product spaces.
%
\end{remark}

\bibliographystyle{siam}
\bibliography{Levy_bib}

\end{document}